\documentclass[a4paper,12pt]{amsart}
\usepackage{amsmath,amssymb,amsthm}
\usepackage[left=2cm,right=2cm]{geometry}
\usepackage{enumerate}
\usepackage{mathtools}
\usepackage{mathtools}
\usepackage{hyperref}
\usepackage{verbatim}

\makeatletter
\@namedef{subjclassname@2020}{%
	\textup{2020} Mathematics Subject Classification}
\makeatother

\newtheorem{thm}{Theorem}

\newtheorem{prop}[thm]{Proposition}
\newtheorem{lem}[thm]{Lemma}
\newtheorem{cor}[thm]{Corollary}

\theoremstyle{definition}
\newtheorem{defn}[thm]{Definition}
\newtheorem{notation}[thm]{Notation}
\newtheorem{OpenProblem}{Open Problem}
\newtheorem{exam}[thm]{Example}

\theoremstyle{remark}
\newtheorem{rem}{Remark}
\def\C{\mathbb C}

\def\CC{\mathcal C}

\def\D{\mathbb D}

\def\Cn{\mathbb{C}^n}

\newcommand{\Supp}{\operatorname{supp}}
\newcommand{\Span}{\operatorname{span}}

\begin{document}
	
	\title[Cyclicity in Poletsky-Stessin Weighted Bergman Spaces]
	{Cyclicity in Poletsky-Stessin Weighted Bergman Spaces}
	
	\author{Pouriya Torkinejad Ziarati}
	
	\address{P.T.Ziarati\\
		Institut de Math\'ematiques de Toulouse; UMR5219 \\
		Universit\'e de Toulouse; CNRS \\
		UPS, F-31062 Toulouse Cedex 9, France} 
        \email{pouriya.torkinejad\_ziarati@math.univ-toulouse.fr}

	\begin{abstract}
            We study the cyclicity of polynomials in Poletsky-Stessin weighted Bergman spaces on various domains in $\mathbb{C}^2$, including the unit ball, the bidisk, and the complex ellipsoid. To this end, we introduce a natural extension of the parameter range for Poletsky-Stessin weighted Bergman spaces on complete Reinhardt domains, yielding a family of spaces that resemble Dirichlet-type spaces on the unit ball. We highlight the differences in the cyclicity behavior of polynomials in these spaces on the bidisk compared to those studied by Bénéteau et al. Finally, we propose several open problems concerning the structure of cyclic polynomials in these spaces.
        \end{abstract}

%	\thanks{The author received support from the University Research School EUR-MINT (State support managed		by the National Research Agency for Future Investments program bearing the reference ANR-18-EURE-0023).}
	
	
	
	\maketitle

\section{Introduction}
\noindent
Given a Banach space $\mathcal{A}$ of holomorphic functions on a domain in $\mathbb{C}^n$ and a function $f \in \mathcal{A}$, $f$ is called \emph{cyclic} when the span of polynomial multiples of $f$ is dense in $\mathcal{A}$. If polynomials are dense in $\mathcal{A}$ and act as multipliers for $\mathcal{A}$—that is, if for any polynomial $p$, we have $ph \in \mathcal{A}$ whenever $h \in \mathcal{A}$—an equivalent characterization of cyclicity is that there exists a sequence of polynomials $p_n$ such that $\|p_n f - 1\|_{\mathcal{A}} \rightarrow 0$. The problem of characterizing cyclic functions is highly space-dependent and remains challenging—even in Hilbert spaces of holomorphic functions on the unit disk. Complete characterizations of cyclic functions are rare. A classical example is given by Beurling's theorem \cite{Beurling1949two}: every nonzero closed shift-invariant subspace of $H^2(\mathbb D)$ is of the form $\theta H^2(\mathbb D)$, where $\theta$ is an inner function. In particular, the cyclic functions in $H^2(\mathbb D)$ are precisely the outer functions. For the Dirichlet space, Brown and Shields \cite{BS84} proved that a cyclic function must be outer and its boundary zero set must have logarithmic capacity zero. They conjectured that these two conditions are also sufficient. This conjecture, posed in 1984, remains open. Recent studies on cyclicity in spaces of holomorphic functions on the unit disk—especially Dirichlet-type spaces—as well as partial results related to the Brown–Shields conjecture, appear in \cite{Beneteau2015_One_Variable,ELFALLAH2016_cyc,ElFallah_BS_Conj}. Additionally, \cite{ElFallah2014_book} provides a comprehensive textbook treatment of Dirichlet-type spaces.

In several variables, a complete characterization of cyclic functions
in Dirichlet-type spaces remains elusive for classical domains such as
the unit ball and the polydisk. However, significant progress has been
made in identifying cyclic polynomials, with explicit characterizations
obtained in the unit ball of dimension two and the bidisk
\cite{Beneteau2016_bidisk,Knese2019_Aniso_Bidisk,Kosinski_Vavitsas2023}.
On the unit ball $\mathbb B_2$, the characterization in
\cite{Kosinski_Vavitsas2023} determines the cyclicity range
of a polynomial with no zeros in the ball according to
whether its boundary zero set is empty, has real dimension
zero (finitely many points), or has real dimension one.

On the bidisk, the characterization in \cite{Beneteau2016_bidisk}
for the classical Dirichlet-type spaces concerns zeros on
the distinguished boundary $\mathbb T^2$. For polynomials
with no zeros in $\mathbb D^2$, a similar distinction applies,
with one important difference: factors of the form
$\zeta-z_i$, $\zeta\in\mathbb T$, have the same cyclicity
range as polynomials with a nonempty zero-dimensional
zero set, although they vanish on one-dimensional subsets
of $\mathbb T^2$. The remaining irreducible polynomials
with one-dimensional zero sets have a smaller cyclicity range.

Weighted Bergman spaces on the unit ball are initially defined via an integral condition and for a limited range of parameters. However, they can be naturally extended to cover the Dirichlet-type spaces of the unit ball, as discussed in \cite{ZhuZhao}. In this work, we adopt the definition introduced by Poletsky and Stessin for weighted Bergman spaces on a hyperconvex domain using an exhaustion function, and we demonstrate that this definition can be extended naturally to a wider range of parameters whenever the domain is complete Reinhardt. On the unit ball, this construction recovers the classical Dirichlet-type spaces, and the corresponding norm varies continuously under continuous changes in the domain, in the sense of varying the exhaustion function. Although it can be shown that on the bidisk this construction does not reproduce the Dirichlet-type spaces studied in \cite{Beneteau2016_bidisk,Bergqvist2018}, an important observation is that our results concerning the cyclicity of polynomials on the bidisk depend solely on the dimension of the zero set on the distinguished boundary—similar to the case of the unit ball \cite{Kosinski_Vavitsas2023}. This contrasts with the spaces in \cite{Beneteau2016_bidisk}, where polynomials of the form $\zeta - z_i$ must be treated separately when categorizing cyclic polynomials whose zero set has dimension one on the distinguished boundary.

Our main cyclicity results are Theorems~\ref{Ellipsoid_Results}
and~\ref{Bidisk_Results}. The first extends the characterization
for polynomials with a nonempty finite boundary zero set
from the unit ball to the complex ellipsoids $\Omega_p$.
The second determines the cyclicity range for polynomials
whose zero set on the distinguished boundary of the bidisk
has dimension one, and gives sufficient and necessary
conditions for cyclicity when this zero set is nonempty
and finite.

The structure of the paper is as follows. In Section~\ref{section: prem}, we recall the original definition of Poletsky and Stessin and apply it to several examples of domains, while also providing the necessary background for the rest of the paper. A precise definition and first properties of cyclicity are given in Section~\ref{subsection: basic results}. Section~{\ref{section: general results}} establishes some general results and extends Poletsky-Stessin weighted Bergman Spaces to more general indices. The motivation here is to construct a family of spaces on complete Reinhardt domains that inherit properties of Dirichlet-type spaces. In Sections~\ref{section: ellipsoid} and \ref{section: bidisk}, we study the cyclicity problem in Poletsky-Stessin weighted Bergman Spaces on the complex ellipsoid and the bidisk, respectively. Finally, in Section~\ref{section: open problem}, we propose several open problems for future research.

\section{Preliminaries} \label{section: prem}
\subsection{Poletsky-Stessin Hardy Spaces}

Throughout this paper, we denote by $\mathcal O(\Omega)$ the space of holomorphic functions on a domain $\Omega\subseteq\mathbb C^n$. We write $A\lesssim B$ if $A\leq CB$ for some constant $C>0$, and $A\asymp B$ if both $A\lesssim B$ and $B\lesssim A$ hold. The dependence of the constants on parameters will be specified whenever necessary. For $z,w\in\mathbb C^n$, we use the standard Hermitian inner product $\langle z,w\rangle=\sum_{j=1}^n z_j\overline{w_j}$ and the associated norm $|z|=\sqrt{\langle z,z\rangle}$. We denote the unit ball by $\mathbb B_n=\{z\in\mathbb C^n:|z|<1\}$ and its boundary by $\mathbb S_n=\partial\mathbb B_n$.
\begin{defn}
    Let $\Omega \subset \mathbb{C}^n$ be a domain. A function 
    $u : \Omega \to [-\infty,0)$ is called an \emph{exhaustion function} of $\Omega$ if
    \[
        \{ z \in \Omega : u(z) < c \} \Subset \Omega
    \]
    for every $c < 0$. In particular, $u(z)\to0$ as $z\to\partial\Omega$.
\end{defn}

For a smooth function $u$ on $\mathbb{C}^n$, we denote
\[
d = \partial + \bar{\partial},
\qquad
d^c = \frac{i}{2\pi}(\bar{\partial}-\partial).
\]
In particular,
\[
dd^c = \frac{i}{\pi}\,\partial\bar{\partial}.
\]

\begin{defn}
Let $u$ be a $\CC^2$ function on a domain $\Omega \subset \mathbb{C}^n$.
The \emph{complex Monge--Ampère operator} of $u$ is the $(n,n)$-form
\[
(dd^c u)^n
= \underbrace{dd^c u \wedge \cdots \wedge dd^c u}_{n \text{ times}} .
\]
\end{defn}

\begin{defn}
Let $\Omega \subset \mathbb{C}^n$ be a domain. A function 
$u : \Omega \to [-\infty,+\infty)$ is called \emph{plurisubharmonic} if $u$ is upper semicontinuous, and for every $z \in \Omega$ and every $\xi \in \mathbb{C}^n$, the function
\[
t \mapsto u(z+t\xi)
\]
is subharmonic on the set $\{t \in \mathbb{C} : z+t\xi \in \Omega\}$.
\end{defn}

\begin{rem}
Let $u$ be a plurisubharmonic function on a domain
$\Omega \subset \mathbb{C}^n$, locally bounded outside
a compact subset of $\Omega$.
The Monge--Ampère operator $(dd^c u)^n$ can be defined
as a positive measure on $\Omega$; see
\cite[Theorem~1.2 and the subsequent paragraph]{Blocki2006}.
We refer to \cite[Chapter~3]{KlimekBook}
for further discussions.
\end{rem}

\begin{defn}
    Let $\Omega \subset \mathbb{C}^n$ be a domain. A plurisubharmonic function 
    $u$ on $\Omega$ is called \emph{maximal} if for every relatively compact 
    open set $U \Subset \Omega$ and every plurisubharmonic function $v$ on $U$ 
    such that
    \[
    v \le u \quad \text{on } \partial U,
    \]
    we have
    \[
    v \le u \quad \text{on } U.
    \]
\end{defn}

\begin{rem}
    If $u \in \CC^2(\Omega)$ is plurisubharmonic, then $u$ is maximal if and only if
    \[
    (dd^c u)^n = 0
    \]
    on $\Omega$. See \cite[Chapter~3]{KlimekBook}.
\end{rem}

\begin{defn}
		A domain $\Omega \subset \mathbb C^n$ is called hyperconvex if there exists a negative plurisubharmonic exhaustion function
		$u:\Omega\to[-\infty,0)$; that is, for every $c<0$,
		\begin{equation*}
			\{z\in\Omega:u(z)<c\}\Subset \Omega.
		\end{equation*}
		See \cite[Chapters~2, 4]{KlimekBook}.
\end{defn}

\begin{defn}
    A domain $\Omega\subset\mathbb C^n$ is called complete Reinhardt
    if, for every $z\in\Omega$ and $w\in\mathbb C^n$,
    \[
    |w_j|\leq |z_j|,\quad j=1,\ldots,n,
    \qquad\text{implies}\qquad w\in\Omega.
    \]
\end{defn}

\begin{rem}\label{rem: logarithmic convexity}
    If $\Omega\subset\mathbb C^n$ is a pseudoconvex Reinhardt
    domain, then its logarithmic image
    \[
        \{(\log|z_1|,\ldots,\log|z_n|):
        z\in\Omega\cap(\mathbb C\setminus\{0\})^n\}
    \]
    is convex. In particular, this holds for the hyperconvex
    complete Reinhardt domains.
    See \cite[Proposition~2]{Zwonek1999}.
\end{rem}

\begin{defn}
If $\Omega \subset \mathbb{C}^n$ is a hyperconvex domain equipped with a plurisubharmonic exhaustion function $u: \Omega \rightarrow [-\infty,0)$, we first define for $r<0$
\begin{equation*}
B_r(u) = \{ z \in \Omega : u(z) < r \}, 
\qquad 
S_r(u) = \{ z \in \Omega : u(z) = r \}.
\end{equation*}

Let $u^{(r)} = \max\{u,r\}$. The associated Monge--Ampère measure is defined by
\begin{equation}
\mu_{u^{(r)}} = (dd^c u^{(r)})^n - \chi_{\Omega \setminus B_r(u)} (dd^c u)^n .
\end{equation}

We then define $\mu_u$ as the weak-$*$ limit of $\mu_{u^{(r)}}$ as $r \to 0^-$. 
The \emph{Poletsky--Stessin Hardy space} on $\Omega$ with respect to $u$ is
\begin{equation}
H^2_u = \{ f \in \mathcal{O}(\Omega) : \mu_u(|f|^2) < +\infty \}.
\end{equation}
\end{defn}

\begin{notation}
        Let $k=(k_1,\dots,k_n)\in\mathbb{Z}_{+}^n$ be a multi-index and let $z=(z_1,\dots,z_n)\in\Cn$.
        We write
        \[
        |k| = k_1 + \cdots + k_n, \qquad
        k! = k_1! \cdots k_n!, \qquad
        z^k = z_1^{k_1} \cdots z_n^{k_n}.
        \]
        If $\gamma=(\gamma_1,\dots,\gamma_n)\in\mathbb{R}^n$, then with a slight abuse of notation we write
        \[
        |z^\gamma| = |z_1|^{\gamma_1} \cdots |z_n|^{\gamma_n}.
        \]
\end{notation}

The following lemma will be used to construct several examples.
\begin{lem} \cite{Rudin2008}
    Let $\sigma$ and $v$ be the normalized surface and volume measures for $\mathbb{S}_n$ and $\mathbb{B}_n$, respectively. Assume that $\gamma \in (-1, +\infty)^n$. Then we have:
    \begin{align}
        \int_{\mathbb{S}_n} |\zeta^\gamma|^2 d\sigma(\zeta) &= \frac{\Gamma(n) \Gamma(\gamma_1 + 1) \dots \Gamma(\gamma_n + 1)}{\Gamma(\gamma_1 + \dots + \gamma_n + n)}, \\[8pt]
        \int_{\mathbb{B}_n} |z^\gamma|^2 dv(z) &= \frac{\Gamma(n+1) \Gamma(\gamma_1 + 1) \dots \Gamma(\gamma_n + 1)}{\Gamma(\gamma_1 + \dots + \gamma_n + n+1)}.
    \end{align}
    Here, $\Gamma$ denotes the Gamma function.
\end{lem}

Here, we compute the norm of monomials in certain function spaces.

\begin{exam} \label{exam: ps Hardy ellipsoid}
    Let $\Omega_P = \{z \in \mathbb{C}^n : |z_1|^{2 p_1 } + \dots + |z_n|^{2 p_n } < 1\}$ where $p_i$'s are positive integers and define
    \[
    u(z) = \frac{1}{2(p_1 \dots p_n)^{1/n}} \log (|z_1|^{2 p_1 } + \dots + |z_n|^{2 p_n }).
    \]
    Note that since $u$ is maximal plurisubharmonic on $\Omega_P\setminus\{0\}$, we have $\mu_{u^{(r)}} = (dd^c u^{(r)})^n$, which is supported on $S_r(u)$. Consequently, one can compute the norm of a monomial $z^L$, where $L=(l_1, \cdots, l_n)$ is a multi-index, as follows:
    \begin{align*}
    \mu_{u}(|z^L|^{2}) 
    &= \lim_{r \to 0^-} \mu_{u^{(r)}}(|z^L|^{2}) \\
    &= \lim_{r \to 0^-} \int_{S_r(u)} |z^L|^{2} \, (dd^c u^{(r)})^n(z).
\end{align*}

    To compute this limit, we use the $(p_1 \cdots p_n)$-sheeted change of variables
    \[
        z_i \mapsto z_i^{p_i}
    \]
    and the fact that $\mu_v = \sigma$, the Lebesgue measure on $\mathbb{S}_n$ (see \cite[Corollary~7.4]{DBook}), for
    \[
        v(z) = \log |z|.
    \]
    This gives
    \begin{align*}
        \mu_{u}(|z^L|^{2})
        &= \frac{p_1 \cdots p_n}{p_1 \cdots p_n} \lim_{r \to 0^-}
        \int_{ S_r(v)}
        |z_1|^{2\frac{l_1}{p_1}} \cdots |z_n|^{2\frac{l_n}{p_n}}
        \, (dd^c v^{(r)})^n(z) \\
        &= 
        \int_{\mathbb{S}_n}
        |z_1|^{2\frac{l_1}{p_1}} \cdots |z_n|^{2\frac{l_n}{p_n}}
        \, d\sigma(z) \\
        &= \frac{\Gamma(n)\Gamma\left(\frac{l_1}{p_1}+1\right)\cdots
        \Gamma\left(\frac{l_n}{p_n}+1\right)}
        {\Gamma\left(\frac{l_1}{p_1}+\cdots+\frac{l_n}{p_n}+n\right)}.
    \end{align*}
\end{exam}

In the previous example, we observe that when $p_i = 1$ for all $i$, we obtain the same norm as in the classical Hardy space on the ball. Similarly, as $p_i \to \infty$ for all $i$, we recover the monomial norm for the classical Hardy space on the polydisc.

\begin{exam}
    Let $\Omega_\lambda = \{z \in \mathbb{C}^2 :  |z_1|<1 ,\ |z_2| <1,\ \lambda |z_1^m z_2^n| <1 \}$ and define
    \begin{equation*}
        u(z) = C \log \max \{ |z_1|, |z_2|, \lambda |z_1^m z_2^n| \},
    \end{equation*}
    where $\lambda > 1$ and $C>0$ is chosen such that $\mu_u (1) = 1$.

    By a similar argument to Example~6.11 in \cite{DBook}, $\mu_u$ is supported on $T_1 \cup T_2$, where
    \begin{equation*}
        T_1 = \{ z \in \mathbb{C}^2 : |z_1| = (1/\lambda)^{1/m} ,\ |z_2| = 1 \}, \quad
        T_2 = \{ z \in \mathbb{C}^2 : |z_1| = 1,\ |z_2| = (1/\lambda)^{1/n} \}.
    \end{equation*}
    We have
    \begin{equation*}
        \mu_u = C_1 \mu_1 + C_2 \mu_2,
    \end{equation*}
    where $\mu_i$ is the uniform probability measure supported on $T_i$. Note that:
    \begin{equation*}
        \mu_u(1) = 1 = C_1 \mu_1(1) + C_2 \mu_2(1) = C_1 + C_2.
    \end{equation*}
    Now, when $z \in \Omega_\lambda$ is sufficiently close to $T_1$, we have:
    \begin{equation*}
        u(z) = u^1(z) = C \log \max \{ |z_2|, \lambda |z_1^m z_2^n| \}.
    \end{equation*}
    Hence, if we set $\tilde{z_1} = z_1^m z_2^n$ and $\tilde{z_2} = z_2$, then:
    \begin{equation*}
        C_1 = \lim_{r \to 0^-} \mu_{(u^1)^{(r)}}(S_r(u^1)) = \lim_{r \to 0^-} \int_{S_r(u^1)} (d d^c (u^1)^{(r)})^2(z) 
        = \lim_{r \to 0^-} m C^2 \int_{S_{\tilde{r}}(V)} (d d^c V^{(\tilde{r})})^2(\tilde{z}) =  m \lambda C^2.
    \end{equation*}
    Here, $V(z) = \log \max \{ \lambda |z_1|, |z_2| \}$, and $\tilde{r} = r / C$. Similarly, we obtain:
    \begin{equation*}
        C_2  = n \lambda C^2.
    \end{equation*}
    Hence we have $C_1 = \frac{m}{m + n}$ and $C_2 = \frac{n}{m + n}$, the norm of the monomial $z_1^k z_2^l$ then can be computed as:
    \begin{equation*}
        \| z_1^k z_2^l \|^2 = \frac{m \lambda^{- \frac{2k}{m}} + n \lambda^{- \frac{2l}{n}} }{m  + n }.
    \end{equation*}
\end{exam}

\subsection{Poletsky-Stessin Weighted Bergman Spaces}

\begin{defn}
    If $\Omega \subset \mathbb{C}^n$ is a hyperconvex domain equipped with a plurisubharmonic exhaustion function $u$, then for $\alpha > -1$, the Poletsky-Stessin weighted Bergman Space on $\Omega$ with respect to $u$ is defined as follows \cite{PS_HardyBergman}:
    \begin{equation}
        A^2_{u,\alpha} = \left\{ f \in \mathcal{O}(\Omega) \; : \; \int_{-\infty}^{0} |r|^\alpha e^r \mu_{u^{(r)}}(|f|^2) \, dr < +\infty \right\}.
    \end{equation}
\end{defn}

The following examples illustrate the definition and motivate the extension
of the parameter range in Section~3 through explicit formulas for monomial
norms. They also introduce the spaces on the complex ellipsoids and the
polydisk that we will study in Sections~4 and~5.

\begin{exam} \label{Omega_P example}
    Let $\Omega_P = \{z \in \mathbb{C}^n : |z_1|^{2 p_1 } + \dots + |z_n|^{2 p_n } < 1\}$ and define
    \begin{equation*}
        u(z) = \frac{1}{2(p_1 \cdots p_n)^{1/n}} \log \left( |z_1|^{2 p_1 } + \dots + |z_n|^{2 p_n } \right).
    \end{equation*}

    Similarly to Example~\ref{exam: ps Hardy ellipsoid}, we use the $p_1 \ldots p_n$-sheeted change of variables $z_i \mapsto z_i^{p_i}$ to compute the norm of a monomial $z^{L}$. Let $p^n = p_1 \cdots p_n$, $v(z') = \frac{1}{2} \log (|z_1|^2 + \dots + |z_n|^2)$, $r' = p r$, and $R = e^{r'}$. We have
    \begin{equation*}
    \begin{aligned}
        \|z^L\|^2_{A^2_{u,\alpha}} 
        &= \int_{-\infty}^{0} |r|^\alpha e^r \mu_{u^{(r)}}(|z^L|^2) \, dr \\
        &= \int_{-\infty}^{0} |r|^\alpha e^r \int_{S_r(u)} |z^L|^2 (dd^c u^{(r)})^n(z) \, dr \\
        &= \frac{p_1 \ldots p_n}{p_1 \ldots p_n} \int_{-\infty}^{0} |r|^\alpha e^r \int_{S_{r'}(v)} |z_1|^{\frac{2 l_1}{p_1}} \cdots |z_n|^{\frac{2 l_n}{p_n}} (dd^c v^{(r')})^n(z) \, dr .
    \end{aligned}
    \end{equation*}
    Again, similarly to Example~\ref{exam: ps Hardy ellipsoid}, $(dd^c v^{(r')})^n = \sigma_R$, the normalized surface measure on $S_{r'}(v) = \left\{ z \in \mathbb{C}^n \mid |z| = R \right\}$. Therefore,
    \begin{equation*}
    \begin{aligned}
        \|z^L\|^2_{A^2_{u,\alpha}}
        &= \int_{-\infty}^{0} |r|^\alpha e^r \int_{S_{r'}(v)} |z_1|^{\frac{2 l_1}{p_1}} \cdots |z_n|^{\frac{2 l_n}{p_n}} \, d\sigma_R(z') \, dr \\
        &= \int_{-\infty}^{0} |r|^\alpha e^r \int_{\mathbb{S}_n} |R\zeta_1|^{\frac{2 l_1}{p_1}} \cdots |R\zeta_n|^{\frac{2 l_n}{p_n}} \, d\sigma(\zeta) \, dr \\
        &= \frac{\Gamma(n) \Gamma \left( \frac{l_1}{p_1} + 1 \right) \cdots \Gamma \left( \frac{l_n}{p_n} + 1 \right)}
        {\Gamma \left( \frac{l_1}{p_1} + \cdots + \frac{l_n}{p_n} + n \right)} 
        \int_{-\infty}^{0} |r|^\alpha e^r R^{2\left(\frac{l_1}{p_1} + \cdots + \frac{l_n}{p_n}\right)} \, dr.
    \end{aligned}
    \end{equation*}
    After the change of variable $r \mapsto -r$, we continue as
    \begin{equation*}
        \begin{aligned}
        \|z^L\|^2_{A^2_{u,\alpha}}
        &= \frac{\Gamma(n) \Gamma \left( \frac{l_1}{p_1} + 1 \right) \cdots \Gamma \left( \frac{l_n}{p_n} + 1 \right)}
        {\Gamma \left( \frac{l_1}{p_1} + \cdots + \frac{l_n}{p_n} + n \right)} 
        \int_{0}^{\infty} r^\alpha e^{-(2p(\frac{l_1}{p_1} + \cdots + \frac{l_n}{p_n})+1)r} \, dr \\
        &= \frac{ \Gamma(n) \Gamma \left( \frac{l_1}{p_1} + 1 \right) \cdots \Gamma \left( \frac{l_n}{p_n} + 1 \right) \Gamma(\alpha+1)}
        {\Gamma \left( \frac{l_1}{p_1} + \cdots + \frac{l_n}{p_n} + n \right)} 
        \left(2p\left(\frac{l_1}{p_1} + \cdots + \frac{l_n}{p_n}\right)+1\right)^{-(\alpha+1)}.
        \end{aligned}
    \end{equation*}
\end{exam}

To make a comparison, we compute the Bergman norm of a monomial in the next example. One may compare the case $\alpha = 0$ in Example~\ref{Omega_P example} with the following.
\begin{exam} [Bergman Space]
    Let $\Omega_P$ be as in Example~\ref{Omega_P example}. We may compute the norm of $z^L$ as follows:
    \begin{equation*}
    \begin{aligned}
        \|z^L\|_{A^2(\Omega_P)}^2 
        &= \int_{\Omega_P} |z^L|^2 \, dv_P(z) \\
        &= \frac{\Gamma \left( \frac{1}{p_1} + \dots + \frac{1}{p_n} + 1 \right)}
        {\Gamma(n+1) \Gamma \left( \frac{1}{p_1} \right) \cdots \Gamma \left( \frac{1}{p_n} \right)}
        \int_{\mathbb{B}_n} |z_1|^{2\left(\frac{l_1+1}{p_1} - 1\right)} 
        \cdots |z_n|^{2\left(\frac{l_n+1}{p_n} - 1\right)} \, dv(z) \\
        &= \frac{\Gamma \left( \frac{1}{p_1} + \dots + \frac{1}{p_n} + 1 \right)}
        {\Gamma \left( \frac{1}{p_1} \right) \cdots \Gamma \left( \frac{1}{p_n} \right)}
        \frac{\Gamma \left( \frac{l_1 + 1}{p_1} \right) \cdots \Gamma \left( \frac{l_n + 1}{p_n} \right)}
        {\Gamma \left( \frac{l_1 + 1}{p_1} + \dots + \frac{l_n + 1}{p_n} + 1 \right)}.
    \end{aligned}
    \end{equation*}
    
    Here, $v_P$ and $v$ denote the normalized volume measures for $\Omega_P$ and $\mathbb{B}_n$, respectively.
\end{exam}

\begin{defn}
    From this point onward, we use the following modified definition of the
    Poletsky--Stessin weighted Bergman norm. For $\alpha>-1$, we set
    \begin{equation} \label{eq: ps modified defn}
        \|f\|^2_{A^2_{u,\alpha}}
        =
        \frac{(2n)^{\alpha+1}}{\Gamma(\alpha+1)}
        \int_{-\infty}^{0}
        |r|^\alpha e^{2nr}\mu_{u^{(r)}}(|f|^2)\,dr.
    \end{equation}
    The subsequent extension to other parameter values is based on this norm.
\end{defn}
This definition yields the classical weighted Bergman Spaces in the case of the unit ball when one chooses $u(z) = \frac{1}{2} \log \left( |z_1|^2 + \dots + |z_n|^2 \right)$ as the exhaustion function.

\begin{exam} \label{exam: modified PS norm on Omega_p}
    Let $\Omega_P$ and $u$ be as in Example~{\ref{Omega_P example}}. Then, the norm of $z^L$ with the modified definition is given by:
    \begin{equation} \label{Ellipsoid_NORM}
    \begin{aligned}
        ||z^L||^2_{A^2_{u,\alpha}} &=\dfrac{ n^{\alpha+1} \Gamma(n) 
        \Gamma \left( \frac{l_1}{p_1} + 1 \right) \cdots 
        \Gamma \left( \frac{l_n}{p_n} + 1 \right)}
        {\Gamma \left( \frac{l_1}{p_1} + \dots + \frac{l_n}{p_n} + n \right)}
         \left( p \left( \frac{l_1}{p_1} + \dots + \frac{l_n}{p_n} \right) + n \right)^{-(\alpha+1)}.
    \end{aligned}
    \end{equation}
    In this way, one may define $A^2_{u,\alpha}(\Omega_P)$ for $\alpha < -1$ using the norm expressed as a series. 
    Moreover, in the case of the unit ball and $\alpha = -(\beta+1)$, these spaces coincide with the Dirichlet-type space $D_{\beta} (\mathbb{B}_n)$ studied in \cite{Kosinski_Vavitsas2023,Vavitsas2}.
\end{exam}

\begin{exam}[Poletsky-Stessin Weighted Bergman Spaces on the Polydisk]
Let $f \in \mathcal{O}(\mathbb{D}^n)$ be a holomorphic function on the polydisk $\mathbb{D}^n$, written in the form of a series as:
\[
f(z) = \sum_{L} a_L z^L.
\]
 Also, let $u(z) = \log \max \{|z_1|,\ldots,|z_n| \}$ be the exhaustion function for the domain. Up to a constant normalization, the norm of $f$ in
$A^2_{u,\alpha}(\mathbb D^n)$ is given by:
\[
||f||^2_{A^2_{u,\alpha}(\mathbb{D}^n)} = \sum_{L} |a_L|^2 (|L| + n)^{-(\alpha+1)}.
\]
Note that other than direct computation this norm can also be obtained by letting $p_i = p \rightarrow \infty$  $ \forall\  1\leq i \leq n$ in \eqref{Ellipsoid_NORM}. For the case of $\alpha \leq -1$, we can take the norm in the form of a power series as the definition of the space.
This is clearly different than the classical weighted Bergman Spaces $A^2_{\alpha}(\mathbb{D}^n)$ given by the norm:
\[
||f||^2_{A^2_{\alpha}(\mathbb{D}^n)} = \sum_{L} |a_L|^2 (l_1 + 1)^{-(\alpha+1)} \ldots (l_n + 1)^{-(\alpha+1)}.
\]
\end{exam}

\subsection{Basic Results of Cyclicity} \label{subsection: basic results}
In this subsection, we present the definition of cyclicity along with fundamental results that are essential for understanding the rest of the paper as we use them several times throughout the paper. While these results hold in more general settings, we formulate them in the context of the function spaces studied here for convenience.

\begin{defn}
    A function $\varphi\in\mathcal O(\Omega)$ is called a
    multiplier of $A^2_{u,\alpha}(\Omega)$ if the operator
    \[
        M_\varphi:A^2_{u,\alpha}(\Omega)\to A^2_{u,\alpha}(\Omega),
        \qquad M_\varphi f=\varphi f,
    \]
    is bounded. We denote the space of multipliers by
    $\operatorname{Mult}(A^2_{u,\alpha})$, equipped with the norm
    \[
        \|\varphi\|_{\operatorname{Mult}(A^2_{u,\alpha})}
        =\|M_\varphi\|.
    \]
\end{defn}

\begin{rem}\label{rem:analytic-multipliers}
For $\alpha\in\mathbb R$, let
$A^2_{u,\alpha}(\Omega)$ denote the spaces obtained by the parameter
extension described in Section~\ref{subsec: Para extension}. If $g$ is holomorphic on a
neighbourhood of $\overline{\Omega}$, then
\[
    g\in\operatorname{Mult}(A^2_{u,\alpha}(\Omega)).
\]
The corresponding statement for the bidisk spaces is proved in
\cite[Proposition~3.4]{BeslikasZiarati2026}. The proof is similar
for the spaces considered here.
\end{rem}

\begin{defn}
    Let $A^2_{u,\alpha}$ be a Poletsky-Stessin weighted Bergman Space on a domain $\Omega \subset \mathbb{C}^n$, a function $f \in A^2_{u,\alpha}$ is called a \emph{shift-cyclic vector} if:
    \begin{equation}
        [f] = \overline{\Span \{ p f : p \in \mathbb{C}[z_1,\ldots,z_n] \}} = A^2_{u,\alpha}.
    \end{equation}
\end{defn}

\begin{rem}
    In this paper, we choose rotation-invariant exhaustion
    functions on complete Reinhardt domains. The associated
    measures are therefore rotation-invariant, and the
    monomials are mutually orthogonal. Consequently, every
    function in $A^2_{u,\alpha}(\Omega)$ can be approximated
    in norm by its Taylor polynomials.

    These spaces are also reproducing kernel Hilbert spaces.
    For $\alpha\geq-1$, boundedness of point evaluations
    follows from \cite[Theorem~3.6]{PS_HardyBergman}.
    For $\alpha<-1$, it follows from the continuous inclusion
    $A^2_{u,\alpha}(\Omega)\subset H^2_u(\Omega)$, obtained
    by comparing the monomial norms in
    \eqref{eq: parameter extention} for the exhaustions
    considered here.

    Polynomial multipliers are dense in
$\operatorname{Mult}(A^2_{u,\alpha})$ in the strong operator
topology. Indeed, rotation invariance and polynomial density
imply that $U_\theta h(z)=h(e^{i\theta}z)$ defines a strongly
continuous unitary group. For a multiplier
$\varphi(z)=\sum_L a_Lz^L$, consider
\[
    p_N(z)=\sum_{|L|\leq N}
    \left(1-\frac{|L|}{N+1}\right)a_Lz^L.
\]
Let
\[
    F_N(\theta)=\frac{1}{N+1}
    \left|\sum_{k=0}^{N}e^{ik\theta}\right|^2
\]
be the Fej\'er kernel
\cite[Chapter~I, Section~2.5]{Katznelson}. For every
$h\in A^2_{u,\alpha}(\Omega)$,
\[
    p_Nh=\frac{1}{2\pi}\int_{-\pi}^{\pi}
    F_N(\theta)\,U_\theta M_\varphi U_{-\theta}h\,d\theta.
\]
Positivity and normalization of $F_N$ give
\[
    \|p_N\|_{\operatorname{Mult}(A^2_{u,\alpha})}
    \leq\|\varphi\|_{\operatorname{Mult}(A^2_{u,\alpha})}.
\]
Since $\theta\mapsto U_\theta M_\varphi U_{-\theta}h$
is norm-continuous,
\cite[Chapter~I, Lemma~2.2]{Katznelson} gives
$p_Nh\to\varphi h$ in norm.
Consequently, for every $f\in A^2_{u,\alpha}(\Omega)$,
\begin{equation} \label{eq:multiplier-cyclic-subspace}
    [f]_{\mathrm{mult}}
    :=\overline{\{\varphi f:
    \varphi\in\operatorname{Mult}(A^2_{u,\alpha})\}}
    =[f].
\end{equation}
\end{rem}

\begin{thm} \label{PEF_non_cyc}
    If $f \in A^2_{u,\alpha}(\Omega)$ vanishes at $\zeta \in \Omega$ then $f$ is not cyclic.
\end{thm}
\begin{proof}
    If we define $\Gamma_\zeta: A^2_{u,\alpha} \rightarrow \mathbb{C}$ by $\Gamma_\zeta(f) = f(\zeta)$, then $\Gamma_\zeta$ is a bounded linear functional and consequently has a closed kernel. For any polynomial $p$, we have $pf \in \operatorname{Ker} \Gamma_\zeta$. This implies:
    \begin{equation*}
        [f] \subseteq  \operatorname{Ker} \Gamma_\zeta \subsetneq A^2_{u,\alpha}
    \end{equation*}
\end{proof}
\begin{rem} \label{Rem:bdd lin fun non cyc}
    In the proof of Theorem~\ref{PEF_non_cyc}, we used the fact that the point evaluation functional $\Gamma_\zeta$ is continuous for any point $\zeta$ in the domain. In fact, the same proof applies whenever the point evaluation functional is bounded at some point in the closure of the domain, which occurs for certain values of $\alpha$ in the spaces we study.
\end{rem}
The following theorem offers a useful sufficient criterion to check whether a function is cyclic. A corresponding result is proved in
\cite[Theorem~4]{Knese2019_Aniso_Bidisk};
we include the proof here for completeness.
\begin{thm} \label{Radial_Dilation}
    Let $f \in A^2_{u,\alpha}(\Omega)$ and assume that $f$ does not vanish on $\Omega$. For $0 < r < 1$, define the radial dilation
    \begin{equation*}
        f_r (z) = f(rz).
    \end{equation*}
    If
    \begin{equation*}
        \sup_{0 < r < 1} \left\| \frac{f}{f_r} \right\|_{A^2_{u,\alpha}} < +\infty,
    \end{equation*}
    then $f$ is cyclic in $A^2_{u,\alpha}(\Omega)$.
\end{thm}

\begin{proof}
Since \(1/f_r\) is holomorphic on a neighbourhood of
\(\overline{\Omega}\), it is a multiplier by
Remark~\ref{rem:analytic-multipliers}, and
\eqref{eq:multiplier-cyclic-subspace} gives
\[
    \left\{\frac{f}{f_r}:0<r<1\right\}\subseteq [f].
\]
By the hypothesis, this family is bounded in
$A^2_{u,\alpha}(\Omega)$. Choose $r_j\uparrow1$. Banach--Alaoglu
Theorem and the reflexivity of $A^2_{u,\alpha}$ yield a weakly convergent sequence $\{ \frac{f}{f_{r_{j_k}}}\}_{k \geq 0}$. Since point
evaluations are bounded and
\[
    \frac{f(z)}{f(r_{j_k}z)}\longrightarrow1
    \qquad (z\in\Omega),
\]
the weak limit is $1$. Thus
\[
    1\in
    \overline{\left\{\frac{f}{f_r}:0<r<1\right\}}^{\,weak}.
\]
Finally, $[f]$ is a norm-closed linear subspace, and hence it is weakly
closed by Mazur's lemma. Therefore, $1\in[f]$, and $f$ is cyclic.
\end{proof}

\begin{lem}\label{lem:product-cyclic}
    Let $f,g\in A^2_{u,\alpha}(\Omega)$ be cyclic.
    If $f$ is a multiplier of $A^2_{u,\alpha}(\Omega)$,
    then $fg$ is also cyclic.
\end{lem}
\begin{proof}
    Since $g$ is cyclic, there exist polynomials $p_j$ such that
    $p_jg\to1$ in $A^2_{u,\alpha}(\Omega)$. Thus
    \[
        \|p_jfg-f\|_{A^2_{u,\alpha}}
        \leq \|f\|_{\operatorname{Mult}(A^2_{u,\alpha})}
        \|p_jg-1\|_{A^2_{u,\alpha}}\longrightarrow0.
    \]
    Hence $f\in[fg]$, and therefore
    $A^2_{u,\alpha}(\Omega)=[f]\subseteq[fg]$.
\end{proof}

\begin{rem} \label{rem: after radial dilation}
    Let $\alpha\geq-1$ and $1\in A^2_{u,\alpha}(\Omega)$.
    If $f,1/f\in H^\infty(\Omega)$, then both are multipliers
    by the norm definition, so
    \eqref{eq:multiplier-cyclic-subspace} gives
    $1=(1/f)f\in[f]$.

    If $f$ is a polynomial with no zeros on
    $\overline{\Omega}$, the same argument applies for
    every $\alpha\in\mathbb R$ in the extended family,
    since $1/f$ is holomorphic on a neighbourhood of
    $\overline{\Omega}$ and is therefore a multiplier by
    Remark~\ref{rem:analytic-multipliers}.
\end{rem}

\section{General Results} \label{section: general results}
In this section, we establish some general results that provide a basis
for studying cyclicity in these spaces. We first give conditions under which the spaces
$A^2_{u,\alpha}(\Omega)$ can be naturally extended to all
$\alpha\in\mathbb R$. We then introduce a capacity condition
for non-cyclicity to be used later.

\subsection{Parameter Extension} \label{subsec: Para extension}

Here, our goal is to extend the definition of Poletsky--Stessin weighted Bergman spaces $A^2_{u,\alpha}(\Omega)$, originally defined for $\alpha \geq -1$, to any $\alpha \in \mathbb{R}$. We introduce the notion of the PSE property (Definition~\ref{Def: PSE prop}), which is a sufficient condition for this extension to be well-defined. Our main result here is the following, which will be proved at the end of this subsection.
\begin{thm} \label{ComReinDef}
    Let $\Omega \subseteq \mathbb{D}^n$ be a hyperconvex complete Reinhardt domain. Then, there exists an exhaustion function $u$ for $\Omega$ such that $A^2_{u,\alpha}(\Omega)$ satisfies the PSE property.
\end{thm}
\begin{defn}
    Let $-1 \leq s,t $, and let $\Omega \subseteq \mathbb{C}^n$ be a hyperconvex domain equipped with an exhaustion function $u$. Then, the operator $R^{s,t}: A^2_{u,t}(\Omega) \to A^2_{u,s}(\Omega)$
    is defined as follows:
    \begin{equation*}
    (R^{s,t} f) (w) = \sum_{L} a_L 
    \frac{\|z^L\|_{A^2_{u,t}}}{\|z^L\|_{A^2_{u,s}}} w^L.
    \end{equation*}
    Here, $f(z) = \sum_{L} a_L z^L \in A^2_{u,t}(\Omega)$. 
\end{defn}
It is straightforward to verify that $R^{s,t}$ is a continuous operator on $A^2_{u,t}(\Omega)$ with a continuous inverse.

\begin{defn} \label{Def: PSE prop}
    We say that the spaces $A^2_{u,\alpha}(\Omega)$, for $-1 \leq \alpha$, satisfy the Parameter Shift Equivalence (PSE) property if for all 
    $\alpha, \alpha', \gamma, \gamma'  \geq -1$ satisfying $\alpha - \alpha' = \gamma - \gamma'$, we have:
    \begin{equation*}
    \frac{\|z^L\|_{A^2_{u,\alpha}}}{\|z^L\|_{A^2_{u,\alpha'}}} 
    \asymp \frac{\|z^L\|_{A^2_{u,\gamma}}}{\|z^L\|_{A^2_{u,\gamma'}}}, 
    \quad \text{as } |L| \to \infty.
    \end{equation*}
\end{defn}
\begin{defn} 
    Let $\{A^2_{u,\alpha}(\Omega) \}_{ \alpha \geq -1}$ be a family of Poletsky-Stessin weighted Bergman spaces satisfying the PSE property. 
    For $ \gamma < -1$, we say that $f \in A^2_{u,\gamma}(\Omega)$ if:
    \begin{equation*}
     R^{{-(1+\gamma)}-1,-1}  f \in A^2_{u,-1}(\Omega) = H^2_u(\Omega).
    \end{equation*}
    The norm of a function $f$ in $A^2_{u,\gamma}(\Omega)$ is defined as:
    \begin{equation} \label{eq: parameter extention}
    \| f \|_{A^2_{u,\gamma}} = \left\|  R^{{-(1+\gamma)}-1,-1}  f \right\|_{A^2_{u,-1}}.
    \end{equation}
\end{defn}

\begin{exam}
    On the unit disk $\D = \{z \in \C : |z|<1\}$ for $f(z) = \sum_{l \geq 0} a_l z^l$ the classical $\alpha$-weighted Bergman norm of $f$, $\alpha > -1$, can be obtained by
    \begin{equation*}
        \| f \|^2_{A^2_{\alpha}(\D)} = \int_{\D} |f(z)|^2 (1-|z|^2)^{\alpha} dv(z) \asymp \sum_{l \geq 0} (l+1)^{-(\alpha +1)} |a_l|^2.
    \end{equation*}
    It can be verified using Example~\ref{exam: modified PS norm on Omega_p} that, for $u(z) = \log|z|$, $A^2_{u,\alpha}(\D) = A^2_{\alpha}(\D)$. Let $s,t > -1$, we have
    \begin{equation*}
        (R^{s,t}f)(z) = \sum_{l \geq 0} a_l (l+1)^{\frac{s-t}{2}} z^l.
    \end{equation*}
    Also, for $\alpha, \alpha', \gamma, \gamma'  \geq -1$ satisfying $\alpha - \alpha' = \gamma - \gamma'$, we can verify
    \begin{equation*}
        \frac{\|z^l\|_{A^2_{\alpha}(\D)}}{\|z^l\|_{A^2_{\alpha'}(\D)}} 
    = (l+1)^{\frac{\alpha'-\alpha}{2}} = (l+1)^{\frac{\gamma'-\gamma}{2}} = \frac{\|z^l\|_{A^2_{\gamma}(\D)}}{\|z^l\|_{A^2_{\gamma'}(\D)}}.
    \end{equation*}
    So the extension in the sense of \eqref{eq: parameter extention} is well-defined and one may define for any $\alpha \in \mathbb{R}$
    \begin{equation*}
        \| f \|^2_{A^2_{\alpha}(\D)} = \sum_{l \geq 0} (l+1)^{-(\alpha +1)} |a_l|^2.
    \end{equation*}
\end{exam}

\begin{prop} \label{Poly_PropX}
    Let $\Omega \subseteq \mathbb{D}^n$ be the polyhedral domain defined by:
    \begin{equation*}
        p_j(z) = \lambda_j |z_1|^{\beta_1^j} \cdots |z_n|^{\beta_n^j} < 1, \quad j = 1, \dots, m,
    \end{equation*}
    where $\lambda_j \geq 1$, $\beta_i^j \in [0,+\infty)$, and $\sum_{i} \beta_i^j = 1$ for all $j$. If $m \geq n$, then the family of spaces $A^2_{u,\alpha}(\Omega)$ satisfies the PSE property when
    \begin{equation*}
        u(z) = \log \max \{ p_j(z)  : j = 1,\dots,m \}.
    \end{equation*}
\end{prop}

\begin{proof}
    Fix $r<0$. Following \cite[Chapter~III, Example~6.11]{DBook},
    we work in local logarithmic coordinates. At $z\in S_r(u)$,
    only the terms satisfying $p_j(z)=e^r$ contribute locally
    to $u^{(r)}$. If the corresponding vectors $\beta^j$ have rank
    less than $n$, a linear change of coordinates shows that
    $u^{(r)}$ depends on fewer than $n$ complex variables, so
    $(dd^c u^{(r)})^n=0$ there.

    Thus $\mu_{u^{(r)}}$ is supported on finitely many tori
    $S_{r,q}$, each determined by $n$ linearly independent
    equations $p_j=e^r$. By rotation invariance,
    $\left.\mu_{u^{(r)}}\right|_{S_{r,q}}=C_q(r)\mu_q^r$,
    where $\mu_q^r$ is the uniform probability measure on
    $S_{r,q}$.
    Since $\sum_i\beta_i^j=1$, we have $p_j(tz)=tp_j(z)$.
    Consequently, for $r,s<0$, the dilation $T(z)=e^{s-r}z$
    maps $S_{r,q}$ onto $S_{s,q}$ and satisfies
    \[
        u^{(s)}\circ T=u^{(r)}+s-r
    \]
    near $S_r(u)$. Since the Monge--Amp\`ere operator commutes
    with biholomorphic pullback, the definition of $\mu_{u^{(r)}}$ gives
    \[
        T_*\mu_{u^{(r)}}=\mu_{u^{(s)}},\qquad
        T_*\mu_q^r=\mu_q^s.
    \]
    Hence $C_q(r)=C_q(s)$, and we may write
    \begin{equation}\label{eq: prop28 mu_r}
        \mu_{u^{(r)}}=\sum_{q=1}^N C_q\mu_q^r,
    \end{equation}
    where $C_q>0$ is independent of $r$.
    The fact that $\mu_{u^{(r)}}(S_r(u))=1$ implies
    \[
        \sum_{q=1}^N C_q=1.
    \]

    Given $S_{r,q}$, choose indices $\sigma_q(j)$,
    $j=1,\ldots,n$, corresponding to linearly independent
    equations defining this torus. To determine $S_{r,q}$,
    we solve the following linear system:
    \[
    \begin{bmatrix}
        \beta_1^{\sigma_q(1)} & \cdots & \beta_n^{\sigma_q(1)} \\
        \vdots & \ddots & \vdots \\
        \beta_1^{\sigma_q(n)} & \cdots & \beta_n^{\sigma_q(n)}
    \end{bmatrix}
    \begin{bmatrix}
        \log |z_1| \\
        \vdots \\
        \log |z_n|
    \end{bmatrix}
    =
    \begin{bmatrix}
        r \\
        \vdots \\
        r
    \end{bmatrix}
    -
    \begin{bmatrix}
        \log \lambda_{\sigma_q(1)} \\
        \vdots \\
        \log \lambda_{\sigma_q(n)}
    \end{bmatrix}.
    \]
    A direct application of Cramer's rule shows that for every $z \in S_{r,q}$
    \[
        |z_i|=C'_{i,q}e^r,
    \]
    where each $C'_{i,q}$ is independent of $r$.
    Thus, for a monomial $z^L$ and $\alpha>-1$,
    \[
    \begin{aligned}
        \int_{-\infty}^{0}|r|^\alpha e^{2nr}
        \mu_q^r(|z^L|^2)\,dr
        &=
        \prod_{i=1}^n(C'_{i,q})^{2l_i}
        \int_{-\infty}^{0}|r|^\alpha e^{2(n+|L|)r}\,dr\\
        &=
        \prod_{i=1}^n(C'_{i,q})^{2l_i}
        \frac{\Gamma(\alpha+1)}
             {[2(|L|+n)]^{\alpha+1}}.
    \end{aligned}
    \]
    Consequently, recalling \eqref{eq: ps modified defn}
    and \eqref{eq: prop28 mu_r},
    \[
        \|z^L\|_{A^2_{u,\alpha}(\Omega)}^2
        =C_L\left(\frac{n}{|L|+n}\right)^{\alpha+1},
        \qquad
        C_L=\sum_{q=1}^N C_q\prod_{i=1}^n(C'_{i,q})^{2l_i}.
    \]
    Since $\|z^L\|_{H^2_u}^2=C_L$, this formula also holds
    for $\alpha=-1$. As $C_L$ is independent of $\alpha$,
    taking ratios establishes the PSE property.
\end{proof}

\begin{proof} [Proof of Theorem~\ref{ComReinDef}]
    Consider the mapping $L:\Omega\cap(\mathbb C\setminus\{0\})^n
\to\Omega'\subset\mathbb R^n$ defined by
    \begin{equation*}
    L(z_1,\ldots,z_n) = (\log |z_1|, \ldots, \log |z_n|).
    \end{equation*}
    For $x = (x^{(1)},\ldots,x^{(n)}) \in \Omega'$, let
    \begin{equation*}
    F_0(x) = \max \{x^{(1)}, \ldots, x^{(n)} \}.
    \end{equation*}
    Given a dense sequence $\{x_i\}_{i \in \mathbb{N}}$ in $\partial \Omega'$, define the iterative sequence:
    \begin{equation*}
        F_i = \max \{ F_{i-1}, f_i \},
    \end{equation*}
    where $f_i(x) = x \cdot v_i + \lambda_i$ with $(1,\ldots,1) \cdot v_i = 1$, and $v_i \in [0,1]^n$ is such that $f_i(x_i) = 0$ and $f_i(x) < 0$ for all $x \in \Omega'$. Then, the sequence $g_i = F_i \circ L$ consists of increasing plurisubharmonic functions on $\Omega$ converging to an exhaustion function $u$ for $\Omega$. So, by Theorem~3.6.1 in \cite{KlimekBook}, we have
    \begin{equation*}
        \| z^L \|_{A^2_{g_i,\alpha}} \to \| z^L \|_{A^2_{u,\alpha}}
    \end{equation*}
    for any multi-index $L$ and $\alpha \in [-1, 0]$. Applying Proposition~\ref{Poly_PropX} completes the proof.
\end{proof}

\begin{notation}
    Theorem~\ref{ComReinDef} allows us to define the Poletsky-Stessin weighted Bergman Spaces on a wide range of domains and parameter values of $\alpha$. From now on, when there is no ambiguity in the choice of the exhaustion function, we denote the space $A^2_{u,\alpha}(\Omega)$ by $D_{\beta}(\Omega)$, where $\beta = -(\alpha+1)$. 
\end{notation}    

\subsection{Kernels and Capacity}

We denote the reproducing kernel of $A^2_{u,\alpha}(\Omega)$
by $k_\alpha$. For the complete Reinhardt domains and exhaustion functions considered here,
the monomials form an orthogonal basis, and therefore
\[
k_\alpha(w,\zeta)
=
\sum_L
\frac{w^L\overline{\zeta}^L}
{\|z^L\|^2_{A^2_{u,\alpha}}},
\qquad w,\zeta\in\Omega.
\]
\begin{defn}
    Let $\mu \in M(\Supp(\mu_u))$, the set of Borel measures on $\Supp(\mu_u)$. The Cauchy transform of $\mu$, denoted by $C[\mu]$, is defined as follows:
    \begin{equation*}
    C[\mu](z) = \int_{\Supp(\mu_u)} k_{-1}(z,\zeta) \, d\mu(\zeta).
    \end{equation*}
\end{defn}

\begin{defn}
    Let $\alpha \in \mathbb{R}$, $E \subseteq \Supp(\mu_u)$, and $\mathcal{P}(E)$ be the set of Borel probability measures supported on $E$. $I_{\alpha}[\mu]$ denotes $\alpha$-energy of $\mu$ and is obtained by
    \begin{equation*}
    I_{\alpha}[\mu] = \iint_{\Supp(\mu_u) \times \Supp(\mu_u)} 
     k_{\alpha}(z,\zeta)  \, d\mu(z) \, d\mu(\zeta).
    \end{equation*}
    The $\alpha$-capacity of $E$ is defined as follows:
    \begin{equation*}
        \operatorname{Cap}_{\alpha}(E) = \inf \{ I_{\alpha}[\mu] : \mu \in \mathcal{P}(E) \}^{-1}.
    \end{equation*}
\end{defn}

The following proposition is the analogue of the duality discussed
in \cite[p.~289, equation~(17)]{BS84}. We include a proof for completeness.
\begin{prop} 
    Let $\{A^2_{u,\alpha}(\Omega)\}_{\alpha \in \mathbb{R}}$ be a family of Poletsky-Stessin weighted Bergman spaces satisfying the PSE property. If $\alpha + \alpha' = -2$, then the dual space satisfies:
    \begin{equation*}
    \left(A^2_{u,\alpha}(\Omega)\right)^* = A^2_{u,\alpha'}(\Omega).
    \end{equation*}
\end{prop}
\begin{proof}
    Since $\alpha+\alpha'=-2$, \eqref{eq: parameter extention} gives
    \[
        \|z^L\|_{A^2_{u,\alpha}}^2
        \|z^L\|_{A^2_{u,\alpha'}}^2=\|z^L\|_{H^2_u}^4.
    \]
    For $f=\sum_L a_Lz^L\in A^2_{u,\alpha}(\Omega)$ and
    $g=\sum_L b_Lz^L\in A^2_{u,\alpha'}(\Omega)$, define
    \[
        \langle f,g\rangle_u
        =\sum_L a_L\overline{b_L}\,\|z^L\|_{H^2_u}^2.
    \]
    By Cauchy--Schwarz,
    \[
        |\langle f,g\rangle_u|
        \leq \|f\|_{A^2_{u,\alpha}}\|g\|_{A^2_{u,\alpha'}}.
    \]
    Conversely, let $\Lambda\in(A^2_{u,\alpha}(\Omega))^*$.
    By the Riesz representation theorem, there exists
    $h=\sum_L c_Lz^L\in A^2_{u,\alpha}(\Omega)$ such that
    \[
        \Lambda(f)=\sum_L a_L\overline{c_L}
        \|z^L\|_{A^2_{u,\alpha}}^2,
        \qquad \|h\|_{A^2_{u,\alpha}}=\|\Lambda\|.
    \]
    Set
    \[
        b_L=c_L\frac{\|z^L\|_{A^2_{u,\alpha}}^2}
                         {\|z^L\|_{H^2_u}^2}.
    \]
    Then
    \[
        \sum_L |b_L|^2\|z^L\|_{A^2_{u,\alpha'}}^2
        =\sum_L |c_L|^2\|z^L\|_{A^2_{u,\alpha}}^2
        =\|\Lambda\|^2.
    \]
    Hence $g=\sum_L b_Lz^L\in A^2_{u,\alpha'}(\Omega)$
    represents $\Lambda$ under the above pairing.
\end{proof}

The proof of the following follows the same approach as \cite[Theorem~5]{BS84};
we give the details in our setting for the reader's convenience.
\begin{thm} \label{capacity condition}
    Let $f \in A^2_{u,\alpha}(\Omega) \cap \CC(\overline{\Omega})$ and assume that $\Omega$ satisfies the PSE property. If $\operatorname{Cap}_{\alpha}(\mathcal{Z}(f) \cap \Supp(\mu_u)) > 0$,
    then $f$ is not cyclic in $A^2_{u,\alpha}(\Omega)$.
\end{thm}
\begin{proof}
    Put $E=\mathcal Z(f)\cap\Supp(\mu_u)$. By the definition
    of capacity, there exists $\mu\in\mathcal P(E)$ such that
    \[
        I_\alpha[\mu]
        =\iint_{E\times E}k_\alpha(z,\zeta)\,
          d\mu(z)\,d\mu(\zeta)<\infty.
    \]
    Expanding the kernel, we obtain
    \[
        I_\alpha[\mu]
        =\sum_L
        \frac{\left|\int_E z^L\,d\mu(z)\right|^2}
             {\|z^L\|_{A^2_{u,\alpha}}^2}.
    \]
    Thus, by Cauchy--Schwarz, for every polynomial $q$,
    \[
        \left|\int_E q\,d\mu\right|^2
        \leq \|q\|_{A^2_{u,\alpha}}^2 I_\alpha[\mu].
    \]
    By polynomial density, $\Lambda(q)=\int_E q\,d\mu$
    extends to a bounded linear functional on
    $A^2_{u,\alpha}(\Omega)$, with $\Lambda(1)=1$.

    Let \(p\) be a polynomial and put \(h=pf\). For each $0<r<1$, we have
    $\Lambda(h_r)=\int_E h_r\,d\mu$.
    Since $h\in\CC(\overline\Omega)$ and $h=0$ on $E$,
    \[
        \Lambda(pf)
        =\lim_{r\to1^-}\Lambda(h_r)
        =\lim_{r\to1^-}\int_E h(r\zeta)\,d\mu(\zeta)
        =0.
    \]
    Hence $[f]\subseteq\ker\Lambda\neq A^2_{u,\alpha}(\Omega)$,
    and $f$ is not cyclic.
\end{proof}
\section{Some Results on the Complex Ellipsoid in $\mathbb{C}^2$} \label{section: ellipsoid}

Here in this section, we study cyclicity properties of some polynomials in $A^2_{u,\alpha}(\Omega_p) = D_{\beta}(\Omega_p)$, $\beta = -(\alpha +1)$, where $\Omega_p = \{(z_1,z_2) \in \mathbb{C}^2 : |z_1|^{2p} + |z_2|^{2p} <1\}$ and $u(z) = \frac{1}{2p} \log (|z_1|^{2 p } + |z_2|^{2 p })$. We recall 
\begin{equation*}
    \| z_1^{j_1} z_2^{j_2} \|^2_{D_{\beta}(\Omega_p)} = \frac{\Gamma\left(\frac{j_1}{p}+1\right) \Gamma\left(\frac{j_2}{p}+1\right)}
        {\Gamma\left(\frac{j_1+j_2}{p}+2\right)} 
        (j_1+j_2+2)^{\beta}.
\end{equation*}
For $p=1$, the domain $\Omega_p$ is the unit ball in $\mathbb C^2$.
The following theorem extends the finite-boundary-zero case of
\cite[Theorem~6]{Kosinski_Vavitsas2023} to $\Omega_p$ for $p\geq1$ and is the main result of this section. We give its proof at the end of the section, after establishing the necessary preliminary results.

\begin{thm} \label{Ellipsoid_Results}
    Let $q$ be a polynomial, non-vanishing on $\Omega_p$, moreover assume that $q$ has a nonempty finite zero set on $\partial \Omega_p$, then $q$ is cyclic on $D_{\beta}(\Omega_p)$ if and only if $\beta \leq 2$.
\end{thm}

\subsection{Properties of the Spaces}
In this subsection, we establish boundedness of boundary point evaluations
for $\beta>2$, which will be used to prove the necessity of $\beta\leq2$
in the main theorem. We first recall the following technical lemma
from \cite{LaplaceMethodBook}, which will be used in the proof of
Proposition~\ref{PEF boundedness for Ellipsoid}.
\begin{lem} [Laplace's Method] \label{Laplace Method}
        Let  $I(\lambda) = \int_a^b e^{-\lambda h(x)} g(x) \,dx$.
        Assume that  $h \in \mathcal{C}^2[a,b]$ and $h \geq 0$  has a unique global minimum at  $x_0 \in (a, b) $, that $ h''(x_0) > 0 $, $ g(x_0) \neq 0 $ and $g$ is continuous on $[a,b]$. Then, as $ \lambda \to \infty $,
        \begin{equation*}
            I(\lambda) \sim \sqrt{\frac{2\pi}{\lambda h''(x_0)}} e^{-\lambda h(x_0)} g(x_0).
        \end{equation*}
    \end{lem}
\begin{prop} \label{PEF boundedness for Ellipsoid}
    Let $\zeta \in \partial \Omega_p$ for $p \geq 1$, then the point evaluation functional $\Gamma_{\zeta} : D_{\beta}(\Omega_p) \rightarrow \mathbb{C}$, $f \mapsto f(\zeta)$ is bounded for $\beta  > 2$.
\end{prop}
\begin{proof}
    Let $J$ be a multi-index $f(z) = \sum_J a_J z^J \in D_{\beta}(\Omega_p)$ and apply Holder inequality on $|\sum_J a_J \zeta^J|^2$ we will have:
    \begin{align*}
        \left| \sum_J a_J \zeta^J \right|^2 &\leq 
        \left( \sum_J |a_J|^2 \frac{\Gamma\left(\frac{j_1}{p}+1\right) \Gamma\left(\frac{j_2}{p}+1\right)}{\Gamma\left(\frac{j_1 + j_2}{p}+2\right)} (|J|+2)^\beta \right) \\
        &\quad \cdot \left( \sum_J \frac{1}{(|J|+2)^\beta} \frac{\Gamma\left(\frac{j_1 + j_2}{p}+2\right)}{\Gamma\left(\frac{j_1}{p}+1\right) \Gamma\left(\frac{j_2}{p}+1\right)} |\zeta^J|^2 \right) \\
        &= \|f\|^2_{D_{\beta}(\Omega_p)} \left( \sum_J \frac{1}{(|J|+2)^\beta} \frac{\Gamma\left(\frac{j_1 + j_2}{p}+2\right)}{\Gamma\left(\frac{j_1}{p}+1\right) \Gamma\left(\frac{j_2}{p}+1\right)} |\zeta^J|^2 \right). \\
        & \asymp \|f\|^2_{D_{\beta}(\Omega_p)} \left( \sum_{j \geq 0} \frac{1}{(j+2)^{\beta - 1}} \sum_{j_1+j_2 = j}\frac{\Gamma\left(\frac{j_1 + j_2}{p}+1\right)}{\Gamma\left(\frac{j_1}{p}+1\right) \Gamma\left(\frac{j_2}{p}+1\right)} |\zeta_1|^{2j_1} |\zeta_2|^{2j_2} \right).
    \end{align*}
    Now the claim is that
    \begin{align} \label{gamma sum in j/p}
        S(j) = \sum_{j_1+j_2 = j} \frac{\Gamma\left(\frac{j_1 + j_2}{p}+1\right)}{\Gamma\left(\frac{j_1}{p}+1\right) \Gamma\left(\frac{j_2}{p}+1\right)} |\zeta_1|^{2j_1} |\zeta_2|^{2j_2} 
        &\leq C,
    \end{align}
    where $C>0$ is a positive constant independent of $j$. Note that this finishes the proof as we will have:
    \begin{align*}
        \left| \sum_J a_J \zeta^J \right|^2 &\lesssim \|f\|^2_{D_{\beta}(\Omega_p)} 
        \left( \sum_{j \geq 0} \frac{1}{(j+2)^{\beta - 1}} 
        \right).
    \end{align*}
    And the sum is finite when $\beta > 2$.

    To prove the claim, we compare the sum in \eqref{gamma sum in j/p} with the following integral:
    \begin{align*}
        I(\lambda) = \int_0^{\lambda} \frac{\Gamma(\lambda+1)}{\Gamma(x+1) \Gamma(\lambda - x+1)} r^x (1-r)^{\lambda - x} \,dx 
        &\lesssim \int_0^{\lambda} \frac{\lambda^\lambda}{x^x (\lambda - x)^{\lambda - x}} \sqrt{\frac{\lambda}{x (\lambda - x)}} r^x (1-r)^{\lambda - x} \,dx \\
        &= \sqrt{\lambda} \int_0^{1} \left( \frac{1}{y^y (1-y)^{1 - y}} r^y (1-r)^{1-y} \right)^{\lambda} \sqrt{\frac{1}{y (1-y)}} \,dy \\
        &= \sqrt{\lambda} \int_0^{1} e^{-\lambda h(y)} g(y) \,dy.
    \end{align*}

    where $\lambda = \frac{j}{p}$, $r = |\zeta_1|^{2p}$, $h(y) = y \log(\frac{y}{r}) + (1-y) \log(\frac{1-y}{1-r})$ and $g(y) = \sqrt{\frac{1}{y(1-y)}}$. Note that in the case of $r = 0$ or $r = 1$, $S(j) = 1$ and we get what we want, so we assume $r \in (0,1)$.
    Let $a_\lambda$ denote the integrand defining $I(\lambda)$.
    By convexity of $\log\Gamma$, $a_\lambda$ is unimodal, so
\[
    \frac1p\sum_{k=0}^j a_\lambda(k/p)
    \leq \int_0^\lambda a_\lambda(x)\,dx
    +\frac2p\max_{[0,\lambda]}a_\lambda.
\]
By Stirling's inequalities, uniformly for $\lambda\geq1$,
\[
    a_\lambda(x)\lesssim
    \sqrt{\frac{\lambda+1}{(x+1)(\lambda-x+1)}}
    e^{-\lambda h(x/\lambda)}\leq1,
\]
since $h\geq0$, with its continuous endpoint values.
Thus $S(j)\lesssim I(\lambda)+1$.
    
    Fix $0<\epsilon<\min\{r,1-r\}$. Since $h$ has its unique
minimum at $r$, with $h(r)=0$ and $h''(r)=g(r)^2$,
Lemma~\ref{Laplace Method} gives
\[
    \sqrt{\lambda}\int_\epsilon^{1-\epsilon}
    e^{-\lambda h(y)}g(y)\,dy\longrightarrow\sqrt{2\pi}.
\]
On the remaining intervals,
$h\geq c:=\min\{h(\epsilon),h(1-\epsilon)\}>0$, so
\[
    \sqrt{\lambda}
    \left(\int_0^\epsilon e^{-\lambda h(y)}g(y)\,dy + \int_{1-\epsilon}^1 e^{-\lambda h(y)}g(y)\,dy \right)
    \leq\sqrt{\lambda}e^{-c\lambda}\int_0^1g(y)\,dy
    \longrightarrow0.
\]
Hence $I(\lambda)$, and therefore $S(j)$, remains bounded.
\end{proof}

\subsection{Proof of Theorem~\ref{Ellipsoid_Results}}
We begin with the following special case, which will be used
in the proof of the main theorem.
\begin{prop} \label{f_cyc_ellipsoid}
    Given $p \geq 1$, the polynomial $f(z) = 1 - r_1^{2p-1} z_1 - r_2^{2p-1} z_2$ where $r_1,r_2 \geq 0$ and $r_1^{2p}+r_2^{2p} = 1$ is cyclic in $D_{\beta}(\Omega_p)$ whenever $\beta \leq 2$.
\end{prop}
\begin{proof}
    \begin{equation*}
    \begin{aligned}
        &\sup_{0<r<1} \left\| \frac{f}{f_r} \right\|^2_{D_{\beta}(\Omega_p)} \\ 
        &=\sup_{0<r<1} \left\| (1 - r_1^{2p-1} z_1 - r_2^{2p-1} z_2) (1 + r(r_1^{2p-1} z_1 + r_2^{2p-1} z_2) + r^2(r_1^{2p-1} z_1 + r_2^{2p-1} z_2)^2 + \ldots)  \right\|^2_{D_{\beta}(\Omega_p)}  \\
        &=\sup_{0<r<1} \left\| 1 + (1-r) \sum_{j \geq 1} r^{j-1} (r_1^{2p-1} z_1 + r_2^{2p-1} z_2)^j  \right\|^2_{D_{\beta}(\Omega_p)}  \\
        &=\sup_{0<r<1} \left\| 1 + (1-r) \sum_{j \geq 1} r^{j-1} \sum_{j_1 + j_2 = j} \binom{j}{j_1} r_1^{(2p-1)j_1} z^{j_1}_1  r_2^{(2p-1)j_2} z_2^{j_2}  \right\|^2_{D_{\beta}(\Omega_p)} \\
        &= 1 + \sup_{0<r<1}(1-r)^2 \sum_{j \geq 1} r^{2(j-1)} 
        \sum_{j_1 + j_2 = j} \binom{j}{j_1}^2 
        \frac{\Gamma\left(\frac{j_1}{p}+1\right) \Gamma\left(\frac{j_2}{p}+1\right)}
        {\Gamma\left(\frac{j}{p}+2\right)} 
        (j+2)^{\beta} r_1^{2(2p-1)j_1} r_2^{2(2p-1)j_2}.
    \end{aligned}
    \end{equation*}
    After applying Stirling's approximation to 
    \[
        \binom{j}{j_1}
        \frac{\Gamma\left(\frac{j_1}{p}+1\right)\Gamma\left(\frac{j_2}{p}+1\right)}
        {\Gamma\left(\frac{j}{p}+2\right)} (j+2),
    \]
    we continue as follows:
    \begin{equation*}
        \begin{aligned}
        &\sup_{0<r<1} \left\| \frac{f}{f_r} \right\|^2_{D_{\beta}(\Omega_p)} \\
        &\asymp 1 + \sup_{0<r<1}(1-r)^2 \sum_{j \geq 1} r^{2(j-1)} (j+2)^{\beta - 1}
        \sum_{j_1 + j_2 = j} \binom{j}{j_1} r_1^{2p j_1} r_2^{2p j_2}
        \frac{j^{j\left(1-\frac{1}{p}\right)}}
        {j_1^{j_1\left(1-\frac{1}{p}\right)} j_2^{j_2\left(1-\frac{1}{p}\right)}} 
        r_1^{2(p-1)j_1} r_2^{2(p-1)j_2} \\
        &= 1 + \sup_{0<r<1}(1-r)^2 \sum_{j \geq 1} r^{2(j-1)} (j+2)^{\beta - 1}
        \sum_{j_1 + j_2 = j} \binom{j}{j_1} r_1^{2p j_1} r_2^{2p j_2}
        \left( \frac{r_1^{2(p-1)\frac{j_1}{j}} r_2^{2(p-1)\frac{j_2}{j}}}
        {\left(\frac{j_1}{j} \right)^{\frac{j_1}{j}\left(1-\frac{1}{p}\right)} \left(\frac{j_2}{j} \right)^{\frac{j_2}{j}\left(1-\frac{1}{p}\right) }}  
        \right)^{j}.
        \end{aligned}
    \end{equation*}

    Note that the function 
    \[
        F(x) = \frac{r_1^{2(p-1)x} r_2^{2(p-1)(1-x)}}
        {x^{x\left(1-\frac{1}{p}\right)} (1-x)^{(1-x)\left(1-\frac{1}{p}\right)}}
    \]
    in $[0,1]$ attains its maximum at $x = r_1^{2p}$. Therefore,
    \begin{equation*}
        \begin{aligned}
            &\sup_{0<r<1} \left\| \frac{f}{f_r} \right\|^2_{D_{\beta}(\Omega_p)} \\
            &\lesssim 1 + \sup_{0<r<1}(1-r)^2 \sum_{j \geq 1} r^{2(j-1)} (j+2)^{\beta - 1}
            \sum_{j_1 + j_2 = j} \binom{j}{j_1}  
            r_1^{2p j_1} r_2^{2p j_2} \\
            &= 1 + \sup_{0<r<1}(1-r)^2 \sum_{j \geq 1} r^{2(j-1)} 
            (j+2)^{\beta - 1} \\
            &\leq 1+\sup_{0<r<1}(1-r)^2
            \sum_{j\geq1}r^{2(j-1)}(j+2)
            =1+\sup_{0<r<1}\frac{3-2r^2}{(1+r)^2}<+\infty,
            \qquad \beta\leq2.
        \end{aligned}
    \end{equation*}
    
      Now Theorem~\ref{Radial_Dilation} finishes the proof.
\end{proof}

\begin{proof}[Proof of Theorem~\ref{Ellipsoid_Results}]
    Let $\{\zeta_j\}_{j=1}^k=\mathcal Z(q)\cap\partial\Omega_p$, and let
\[
    s_j(z)=1-\sum_{\ell=1}^2
    |\zeta_{j,\ell}|^{2p-2}\overline{\zeta_{j,\ell}}z_\ell,
\]
with zero coefficients at zero coordinates. Up to coordinatewise
rotations, these are the polynomials in
Proposition~\ref{f_cyc_ellipsoid}, vanishing on $\partial\Omega_p$
only at $\zeta_j$. Since $s_j(\zeta_j)=0$, Cauchy--Schwarz gives
\[
    \operatorname{dist}(z,\zeta_j)\geq C_j|s_j(z)|,
    \qquad
    C_j=\left(\sum_{\ell=1}^2
    |\zeta_{j,\ell}|^{4p-2}\right)^{-1/2}.
\]
On the other hand, Łojasiewicz's inequality \cite{KrantzParks2002}
gives a positive integer $m$ and a constant $C>0$ such that
\[
    |q(z)|\geq C\operatorname{dist}
    (z,\mathcal Z(q)\cap\partial\Omega_p)^m
\]
for any $z\in\partial\Omega_p$. Choosing $\zeta_j$ nearest to $z$
and bounding the remaining factors by their supremum norms,
we obtain
\[
    |q(z)|\geq C'\prod_{j=1}^k|s_j(z)|^m,
    \qquad
    C'=C\left(\min_{1\leq j\leq k}
    \frac{C_j}{\prod_{i\ne j}\|s_i\|_{\infty}}\right)^m>0,
\]
for any $z\in\partial\Omega_p$, where the supremum norms
are taken over $\partial\Omega_p$.
     By increasing $m$ sufficiently, we may assume that 
    \[
    Q(z) = \frac{\prod_{j = 1}^k s_j(z)^m}{q(z)} \in D_\beta(\Omega_p).
    \]
    Assume $\beta\leq2$. By Proposition~\ref{f_cyc_ellipsoid},
each $s_j$ is cyclic in $D_\beta(\Omega_p)$. Since the $s_j$ are cyclic polynomials and hence multipliers,
Lemma~\ref{lem:product-cyclic} implies that
$qQ=\prod_{j=1}^k s_j^m$ is cyclic. Fix $\epsilon > 0$, and let $p_1$ and $p_2$ be polynomials such that
    \[
    \| p_1 q Q - 1 \|_{D_\beta} < \frac{\epsilon}{2}
    \]
    and
    \[
    \| p_2 - Q \|_{D_\beta} < \| p_1 q \|^{-1}_{\operatorname{Mult}(D_\beta)} \frac{ \epsilon }{2},
    \]
    where $\| p_1 q \|_{\operatorname{Mult}(D_\beta)}$ denotes the multiplier norm of $p_1 q$, and the inequality follows from the density of polynomials in $D_\beta$. Then, we have
    \begin{equation*}
        \| p_1 p_2 q - 1 \|_{D_\beta} 
        \leq \| p_1 p_2 q - p_1 q Q \|_{D_\beta} + \| p_1 q Q - 1 \|_{D_\beta} 
        < \epsilon.
    \end{equation*}

    Conversely, if $\beta>2$, choose
$\zeta\in\mathcal Z(q)\cap\partial\Omega_p$.
By Proposition~\ref{PEF boundedness for Ellipsoid},
evaluation at $\zeta$ is bounded on $D_\beta(\Omega_p)$.
Therefore,
$q$ is not cyclic by Remark~\ref{Rem:bdd lin fun non cyc}.
\end{proof}

\section{Some Results on the Bidisk} \label{section: bidisk}
In this section, we study cyclicity properties of polynomials in $A^2_{u,\alpha}(\D^2) = D_{\beta}(\D^2)$, $\beta = -(\alpha +1)$, where $u(z) = \log \max \{|z_1|,|z_2| \}$. We recall 
\begin{equation*}
    \| z_1^{j_1} z_2^{j_2} \|^2_{D_{\beta}(\D^2)} =  
        (j_1+j_2+2)^{\beta}.
\end{equation*}  

The following is the main result of this section:
\begin{thm} \label{Bidisk_Results}
    Let $f$ be a polynomial that does not vanish on $\mathbb{D}^2$. Then:
    \begin{enumerate}[(a)]
        \item \label{Bidisk_Results_a} If $\mathcal Z(f)\cap\mathbb T^2$ is nonempty and finite, then $f$ is cyclic in $D_{\beta}(\mathbb{D}^2)$ for $\beta \leq 3/2$, and non-cyclic for $\beta > 2$.
        \item \label{Bidisk_Results_b} If $\dim(\mathcal{Z}(f) \cap \mathbb{T}^2) = 1$, then $f$ is cyclic in $D_{\beta}(\mathbb{D}^2)$ if and only if $\beta \leq 1$.
        \item \label{Bidisk_Results_c} For $\beta > 2$, $f$ is cyclic in $D_{\beta}(\mathbb{D}^2)$ if and only if it does not vanish on $\overline{\mathbb{D}}^2$.
    \end{enumerate}
\end{thm}
The proof is given in Subsection~5.2.

On the polydisk, the classical weighted Bergman spaces $A^2_{\alpha}(\D^n)$, $\alpha \geq -1$, are defined by prescribing the norm of monomials $z^L$ as
\begin{equation*}
    \| z^L \|^2_{A^2_\alpha(\D^n)} = (l_1+1)^{-(\alpha+1)} \ldots (l_n+1)^{-(\alpha+1)}.
\end{equation*}
We can also extend the definition to any $\alpha \in \mathbb{R}$ using the power series norm. In this case, we obtain the classical Dirichlet-type spaces $\mathfrak{D}_{\beta}(\D^n) = A^2_{-(\beta +1)}(\D^n)$. B\'en\'eteau et al.~\cite{Beneteau2016_bidisk} gave a full characterization of cyclic polynomials for these spaces of dimension $2$.

Comparing our main result on the bidisk with the one in \cite{Beneteau2016_bidisk}, we observe that, in our case, the cyclicity of a polynomial is completely determined by the dimension of its zero set on the distinguished boundary. In contrast, in \cite{Beneteau2016_bidisk}, polynomials of the form $\zeta - z_i$ exhibit a different cyclicity behavior compared to other polynomials that vanish at infinitely many points on $\mathbb{T}^2$.

\subsection{Properties of the Spaces}
We compare these spaces with the classical weighted Bergman spaces
and establish some properties needed in the proof of Theorem~\ref{Bidisk_Results}.

\begin{prop} \cite[Proposition~3.5]{PS_HardyBergman} \label{PS_Inc_Prop}
    Let $u$ and $v$ be two continuous plurisubharmonic exhaustion functions for the domain $\Omega \subset \mathbb{C}^n$ and assume that $(dd^c u)^n$ is compactly supported, then $A^2_{v,\alpha}(\Omega) \subset A^2_{u,\alpha}(\Omega)$ for any $\alpha \geq -1$.
\end{prop}

As a corollary of Proposition~\ref{PS_Inc_Prop}, we show that one cannot obtain the classical weighted Bergman spaces on the polydisk using the Poletsky--Stessin method.

\begin{cor}
    There exists no continuous plurisubharmonic exhaustion function $v$ for $\mathbb{D}^n$, $n \geq 2$ such that $A^2_{v,\alpha}(\mathbb{D}^n) = A^2_{\alpha}(\mathbb{D}^n)$.
\end{cor}

\begin{proof}
    Let $\epsilon > 0$ and take $u(z) = \log \max \{|z_1|,\ldots,|z_n| \}$. Also, let $f(z) = \sum_{|L|\geq 0} (l_1+1)^{\frac{\alpha-\epsilon}{2}} \ldots (l_n+1)^{\frac{\alpha-\epsilon}{2}} z^L$. Then one can write:
    \begin{equation*}
        \begin{split}
            ||f||^2_{A^2_{\alpha}(\mathbb{D}^n)} = \sum_{L} \frac{1}{(l_1+1)^{1+\epsilon}} \ldots \frac{1}{(l_n+1)^{1+\epsilon}} < \infty.
        \end{split}
    \end{equation*}
    So $f \in A^2_{\alpha}(\mathbb{D}^n)$. Also:
    \begin{equation} \label{PSNorm_eq1}
        \begin{aligned}
            ||f||^2_{A^2_{u,\alpha}(\mathbb{D}^n)} = \sum_L \frac{(l_1+1)^{\alpha-\epsilon} \ldots (l_n+1)^{\alpha-\epsilon}}{(|L|+n)^{\alpha + 1}} &\geq \\ \sum_{l \geq 0} \frac{(l+1)^{\alpha-\epsilon} \ldots (l+1)^{\alpha-\epsilon}}{(n l+n)^{\alpha + 1}} &\gtrsim \\ \sum_{l \geq 0} (l+1)^{(n-1)\alpha - n \epsilon - 1}&.
        \end{aligned}
    \end{equation}
    Note that the right-hand side of \eqref{PSNorm_eq1} diverges whenever $\alpha>0$ and $\epsilon = \frac{n-1}{n} \alpha$, so for such a case $f \not\in A^2_{u,\alpha}(\mathbb{D}^n)$. Now Proposition~\ref{PS_Inc_Prop} finishes the proof.
\end{proof}

\begin{prop} \label{PEF_Bidisk}
    Let $\zeta \in \mathbb{T}^n$, then the point evaluation functional $\Gamma_{\zeta} : 
    A^2_{u,\alpha}(\mathbb{D}^n) = D_{\beta}(\mathbb{D}^n) \rightarrow \mathbb{C}$, $f \mapsto f(\zeta)$ is bounded for $\beta  > n$.
\end{prop}
\begin{proof}
    Let $J$ be a multi-index $f(z) = \sum_J a_J z^J \in D_{\beta}(\mathbb{D}^n)$ and apply H\"older inequality on $|\sum_J a_J \zeta^J|^2$ we will have:
    \begin{align*}
        \left| \sum_J a_J \zeta^J \right|^2 &\leq 
        \left( \sum_J |a_J|^2 (|J|+n)^\beta \right) \cdot \left( \sum_J \frac{1}{(|J|+n)^\beta}  |\zeta^J|^2 \right) \\
        &= \|f\|^2_{D_{\beta}(\mathbb{D}^n)} \left( \sum_J \frac{1}{(|J|+n)^\beta}  \right). \\
    \end{align*}
    The last sum on the right-hand side is finite. Indeed, grouping the terms according to $j=|J|$, we obtain
\[
    \sum_{J\in\mathbb N^n}\frac{1}{(|J|+n)^\beta}
    =
    \sum_{j\geq0}
    \binom{j+n-1}{n-1}(j+n)^{-\beta}
    \lesssim
    \sum_{j\geq1}j^{n-1-\beta}<+\infty,
\]
since $\beta>n$. This completes the proof.
\end{proof}
\begin{prop} \label{Algebra_Prop}
$A^2_{u,\alpha}(\mathbb{D}^n) = D_{\beta}(\mathbb{D}^n)$ is an algebra when $\beta = -(\alpha + 1) > n$.
\end{prop}

\begin{proof}
Let $f(z) = \sum_{|L| \geq 0} a_L z^L$ and $g(z) = \sum_{|L| \geq 0} b_L z^L$ be holomorphic functions in $D_{\beta}(\mathbb{D}^n)$. Their product is:
\[
f(z)g(z) = \sum_{|L| \geq 0} \sum_{L_1 + L_2 = L} a_{L_1} b_{L_2} z^L.
\]
The squared norm of \( f \cdot g \) is:
\[
\|f \cdot g\|_{D_{\beta}(\mathbb{D}^n)}^2 = \sum_{|L|\geq 0} (|L| + n)^{\beta} \left| \sum_{L_1 + L_2 = L} a_{L_1} b_{L_2} \right|^2.
\]
Applying Hölder’s inequality:
\begin{align*}
\left| \sum_{L_1 + L_2 = L} a_{L_1} b_{L_2} \right|^2 
&\leq \left( \sum_{L_1 + L_2 = L} \frac{1}{(|L_1| + n)^{\beta} (|L_2| + n)^{\beta}} \right) \\
&\quad \cdot \left( \sum_{L_1 + L_2 = L} |a_{L_1}|^2 (|L_1| + n)^{\beta} |b_{L_2}|^2 (|L_2| + n)^{\beta} \right).
\end{align*}
Thus:
\begin{equation*}
\begin{aligned}
\|f \cdot g\|^2_{D_{\beta}(\mathbb{D}^n)} 
&\leq \sum_{L} (|L| + n)^{\beta} 
 \left( \sum_{L_1 + L_2 = L} \frac{1}{(|L_1| + n)^{\beta} (|L_2| + n)^{\beta}} \right) 
 \left( \sum_{L_1 + L_2 = L} |a_{L_1}|^2 (|L_1| + n)^{\beta} |b_{L_2}|^2 (|L_2| + n)^{\beta} \right) \\
&\leq \sum_{|L| \geq 0} \left( \frac{|L| + n}{|L| + 2n} \right)^{\beta} 
 \sum_{L_1 + L_2 = L} \left( \frac{1}{(|L_1| + n)} + \frac{1}{(|L_2| + n)} \right)^\beta \\& \cdot  \left( \sum_{L_1 + L_2 = L} |a_{L_1}|^2 (|L_1| + n)^{\beta} |b_{L_2}|^2 (|L_2| + n)^{\beta} \right) \\
&\leq \sum_{|L| \geq 0} \sum_{L_1 + L_2 = L} \left( \frac{1}{(|L_1| + n)} + \frac{1}{(|L_2| + n)} \right)^\beta  \left( \sum_{L_1 + L_2 = L} |a_{L_1}|^2 (|L_1| + n)^{\beta} |b_{L_2}|^2 (|L_2| + n)^{\beta} \right) \\
&\leq C_{\beta} \sum_{|L| \geq 0} \sum_{L_1 + L_2 = L} \left( \frac{1}{(|L_1| + n)^\beta} + \frac{1}{(|L_2| + n)^\beta} \right) \cdot \left( \sum_{L_1 + L_2 = L} |a_{L_1}|^2 (|L_1| + n)^{\beta} |b_{L_2}|^2 (|L_2| + n)^{\beta} \right) \\
&\leq 2 C_{\beta} \left( \sum_{|L| \geq 0}  \frac{1}{(|L| + n)^\beta} \right) \cdot \left( \sum_{|L| \geq 0} \sum_{L_1 + L_2 = L} |a_{L_1}|^2 (|L_1| + n)^{\beta} |b_{L_2}|^2 (|L_2| + n)^{\beta} \right) \\
& \leq 2 C_{\beta} \sum_{|L| \geq 0}  \frac{1}{(|L| + n)^\beta}  \|f\|_{D_{\beta}(\mathbb{D}^n)}^2 \|g\|_{D_{\beta}(\mathbb{D}^n)}^2,
\end{aligned}
\end{equation*}
where $C_\beta = 2^{\beta - 1}$ is obtained from Jensen's inequality. So for \(\beta > n\), \( f \cdot g \in A^2_{u,\alpha}(\mathbb{D}^n) \).
\end{proof}

\subsection{Proof of Theorem~\ref{Bidisk_Results}}
We first prove part~\eqref{Bidisk_Results_c}, then turn to parts~\eqref{Bidisk_Results_a} and~\eqref{Bidisk_Results_b}.

\begin{prop} \label{prop: part c}
    For $\beta > 2$, a polynomial $f$ is cyclic in $D_\beta(\D^2)$ if and only if $f$ does not vanish on $\overline{\D}^2$.
\end{prop}
\begin{proof}
    Proposition~\ref{PEF_Bidisk} implies that for $\beta > 2$ the point evaluation functional is bounded at any point of $\overline{\D}^2$. Therefore, by Remark~\ref{Rem:bdd lin fun non cyc}, a cyclic function in $D_\beta(\D^2)$ cannot have any zero on $\overline{\D}^2$.

    Conversely, if $f$ does not vanish on $\overline{\D}^2$,
    then $f$ and $1/f$ are multipliers. Thus
    $1=(1/f)f\in[f]$, and $f$ is cyclic
    (see Remark~\ref{rem: after radial dilation}).
\end{proof}

\begin{prop} \label{cuma lemma}
The function \( f_0(z) = 1 - \frac{z_1 + z_2}{2} \) is cyclic in \( A^2_{u, \alpha}(\mathbb{D}^2) = D_{\beta}(\mathbb{D}^2) \) when \( \beta = -(\alpha+1) \leq \frac{3}{2} \). %and non-cyclic otherwise.%
\end{prop}

\begin{proof}
Let $F:\mathbb{D}^2 \rightarrow \mathbb{D}$ be $(z_1,z_2) \mapsto \frac{z_1+z_2}{2}$ and define $C_F:D_{\beta - \frac{1}{2}}(\mathbb{D}) \rightarrow D_{\beta }(\mathbb{D}^2)$ as $f \mapsto f \circ F$. Then for $f(z) = \sum_{j \geq 0} a_j z^j$ we may write:
\begin{equation} \label{C_F closedness}
    \begin{aligned}
        \left\| C_F(f) \right\|^2_{D_{\beta }(\mathbb{D}^2)} 
        &= \sum_{j \geq 0} |a_j|^2 \frac{1}{2^{2j}} \| (z_1 + z_2)^j \|_{D_{\beta }(\mathbb{D}^2)}^2 \\
        &=  \sum_{j \geq 0} |a_j|^2 \frac{1}{2^{2j}} (j+2)^\beta \sum_{j_1 = 0}^j \binom{j}{j_1}^2 \\
        &=  \sum_{j \geq 0} |a_j|^2 \frac{1}{2^{2j}} \binom{2j}{j} (j+2)^\beta  \\
        &\asymp \sum_{j \geq 0} |a_j|^2 \frac{1}{2^{2j}} \frac{2^{2j}}{(j+2)^{\frac{1}{2}}} (j+2)^\beta \\
        &= \sum_{j \geq 0} |a_j|^2 (j+2)^{\beta - \frac{1}{2}} \\
        &\asymp \| f \|_{D_{\beta - \frac{1}{2}}(\mathbb{D})}^2.
    \end{aligned}
\end{equation}
This proves that $C_F$ is bounded. Moreover, $B = \operatorname{Im}(C_F)$ is a closed subspace of $D_{\beta}(\mathbb{D}^2)$. This implies the cyclicity of $f_0$ whenever $\beta \leq \frac{3}{2}$, since it is the image of the function $g(z) = 1 - z$ under $C_F$, which is cyclic in $D(\mathbb{D})$. In this case, there exists a sequence of polynomials $\{p_n\}$ such that $p_n g \to 1$ in $D(\mathbb{D})$, and the corresponding sequence $\{q_n = C_F(p_n)\}$ can be used to approximate $1$ via $q_n f_0$ in $D_{\beta}(\mathbb{D}^2)$.
\end{proof}

\begin{rem}
    The function $f(z)=1-\frac{z_1+z_2}{2}$ is noncyclic
    whenever $\beta>2$, since it vanishes at $(1,1)$ and
    point evaluations are bounded on $\overline{\mathbb D^2}$
    by Proposition~\ref{PEF_Bidisk}.
    Proposition~\ref{cuma lemma} establishes its cyclicity
    for $\beta\leq\frac32$, while the case
    $\frac32<\beta\leq2$ remains open.

    The proof of Proposition~\ref{f_cyc_ellipsoid} cannot be
    applied to the bidisk by letting $p\to\infty$, since
    its estimates are not uniform in $p$.
\end{rem}

\begin{proof}[Proof of Theorem~\ref{Bidisk_Results}\,\textup{(a)}]
    Let $f$ be a polynomial with no zeros in $\mathbb D^2$
    and a nonempty finite zero set
    $\{\zeta_j\}_{j=1}^k$ on $\mathbb T^2$. Put
    \[
        s_j(z)=1-\frac{\overline{\zeta_{j,1}}z_1+
        \overline{\zeta_{j,2}}z_2}{2}.
    \]
    By rotation invariance and Proposition~\ref{cuma lemma},
    each $s_j$ is cyclic for $\beta\leq\frac32$.

    As explained in \cite[Section~2.2]{Beneteau2016_bidisk},
a zero of $f$ on $\partial\mathbb D^2\setminus\mathbb T^2$
would force a factor $z_i-\zeta$, with $\zeta\in\mathbb T$,
contradicting finiteness of $\mathcal Z(f)\cap\mathbb T^2$.
Thus
\[
    \mathcal Z(f)\cap\overline{\mathbb D^2}
    =\{\zeta_j\}_{j=1}^k.
\]
    Łojasiewicz's inequality \cite{KrantzParks2002} gives
    \[
        |f(z)|\gtrsim
        \operatorname{dist}(z,\{\zeta_j\}_{j=1}^k)^N,
        \qquad z\in\overline{\mathbb D^2},
    \]
    for some $N>0$. Since $|s_j(z)|\lesssim|z-\zeta_j|$,
    choosing $m$ sufficiently large ensures that
    \[
        Q=\frac{\prod_{j=1}^k s_j^m}{f}
    \]
    and its first partial derivatives are bounded on
    $\mathbb D^2$, by the quotient rule and the preceding
    estimate. Hence $Q\in D_2(\mathbb D^2)\subseteq
    D_\beta(\mathbb D^2)$ for $\beta\leq\frac32$.

    By Lemma~\ref{lem:product-cyclic}, $fQ=\prod_j s_j^m$
    is cyclic. Choose polynomials $p_n\to Q$ in $D_\beta(\mathbb D^2)$.
Since $f$ is a multiplier, $fp_n\to fQ$, so
\[
    D_\beta(\mathbb D^2)=[fQ]\subseteq[f].
\]

    For $\beta>2$, evaluation at any $\zeta_j$ is a bounded
    nonzero functional annihilating $[f]$, by
    Proposition~\ref{PEF_Bidisk}. Thus $f$ is not cyclic.
\end{proof}

\begin{prop} \label{1-z_Prop}
    If $f(z) = \zeta - z_i,\ i=1,2$ where $\zeta \in \mathbb{T}$ then $f$ is cyclic in $D_{\beta }(\mathbb{D}^2)$ if and only if $\beta \leq 1$.
\end{prop}
\begin{proof}
    Without loss of generality, we can prove the statement only for $f(z) = 1 - z_1$. Assume that $\beta \leq 1$. To prove the cyclicity of $f$, as in Proposition~\ref{f_cyc_ellipsoid}, we need to estimate the norm of $f/f_r$ and then apply Theorem~\ref{Radial_Dilation}. 
    \begin{equation*}
        \sup_{0<r<1} \left\| \frac{f}{f_r} \right\|^2_{D_{\beta }(\mathbb{D}^2)}
        = 1 + \sup_{0<r<1} (1-r)^2 \sum_{j \geq 1} (j+2)^{\beta} r^{2j - 2} 
        \asymp 1 + \sup_{0<r<1} (1-r)^{1 - \beta} < +\infty.
    \end{equation*}
    Now let $\beta > 1$, our goal is to show that the capacity of $\mathcal{Z}(f) \cap \mathbb{T}^2$ is positive to apply Theorem~\ref{capacity condition}. Assume that $\mu$ is the uniform probability measure on $\{1\}\times \mathbb{T} = \mathcal{Z}(f) \cap \mathbb{T}^2$ we will have:
    \begin{equation*}
        I_{\beta}[\mu] = \iint_{\mathbb{T}^2 \times \mathbb{T}^2} k_{\beta}(z,w) d\mu(z)d\mu(w) = \iint_{\mathbb{T}\times\mathbb{T}} \sum_{|J| \geq 0} \frac{z_2^{j_2} \bar{w}_2^{j_2}}{(|J|+2)^\beta} d\mu(z_2) d\mu(w_2) = \sum_{j_1 \geq 0} \frac{1}{(j_1+2)^{\beta}} < +\infty
    \end{equation*}
    This implies that the capacity of $\mathcal{Z}(f) \cap \mathbb{T}^2$ is positive.
\end{proof}

\begin{proof}[Proof of Theorem~\ref{Bidisk_Results}\,\textup{(b)}]
    Recall that the classical spaces $\mathfrak D_\alpha(\mathbb D^2)$
    have squared monomial norms
    \[
        \|z_1^{j_1}z_2^{j_2}\|_{\mathfrak D_\alpha}^2
        =(j_1+1)^\alpha(j_2+1)^\alpha.
    \]
    By \cite[Main Theorem]{Beneteau2016_bidisk}, an irreducible
    polynomial depending on both variables, non-vanishing in
    $\mathbb D^2$ and vanishing at infinitely many points of
    $\mathbb T^2$, is cyclic in $\mathfrak D_\alpha$ if and only if
    $\alpha\leq\frac12$. We adapt their argument to $D_\beta$.

    Assume $\beta\leq1$. Factors with a nonempty finite zero set
    on $\mathbb T^2$ are cyclic by part~\eqref{Bidisk_Results_a}.
    Factors with no zeros on $\mathbb T^2$ have no zeros on
    $\overline{\mathbb D^2}$, by
    \cite[Section~2.2]{Beneteau2016_bidisk}, and are cyclic by
    part~\eqref{Bidisk_Results_c} and the inclusions between the spaces.
    Thus, by Lemma~\ref{lem:product-cyclic}, it remains to consider
    irreducible factors with infinitely many zeros on $\mathbb T^2$.
    The case $\zeta-z_i$ has already been treated, so assume
    that $f$ depends on both variables.

    Let $(n,m)$ be the bidegree of $f$. By
    \cite[Section~2.2 and Proposition~4.2]{Beneteau2016_bidisk},
    there exist a unitary $(n+m)\times(n+m)$ matrix $U$ and
    a column vector polynomial $B\in\mathbb C^{n+m}[z_1,z_2]$
    such that
    \[
        (I-A(z))B(z)\in f(z)\mathbb C^{n+m}[z_1,z_2],
        \qquad
        A(z)=U
        \begin{bmatrix}
            I_nz_1&0\\
            0&I_mz_2
        \end{bmatrix}.
    \]
    Moreover, there exists a row vector polynomial $a$ such that
    $p(z_2)=a(z)B(z)$ is a one-variable polynomial
    with no zeros in $\mathbb D$.

    Let $g\in[f]^\perp$ and let $v\in\mathbb C^{n+m}[z_1,z_2]$
    be a row vector polynomial. Since
    \[
        (I-A^k)B=\sum_{\ell=0}^{k-1}A^\ell(I-A)B
        \in f\mathbb C^{n+m}[z_1,z_2],
    \]
    we obtain
    \[
        \langle g,vB\rangle=\langle g,vA^kB\rangle,
        \qquad k\geq0.
    \]
    Put $d=\deg v+\deg B+1$, where degrees are total degrees,
    taking the maximum over the entries.
    Since $A$ is homogeneous of degree one, the monomials in
    $vA^{kd}B$ have degrees in $[kd,(k+1)d)$.
    Hence $\{vA^{kd}B\}_{k\geq0}$ is orthogonal.
    Moreover, $A$ is unitary on $\mathbb T^2$, so
    \[
        \|vA^{kd}B\|_{D_\beta}^2
        \lesssim (k+1)^\beta
        \|vA^{kd}B\|_{H^2(\mathbb D^2)}^2
        \lesssim (k+1)^\beta,
    \]
    where the last estimate follows from
    $|vA^{kd}B|\leq|v||B|$ on $\mathbb T^2$,
    with Euclidean norms on the vectors.

    If some $vA^{kd}B=0$, then $\langle g,vB\rangle=0$.
    Otherwise, Bessel's inequality gives
    \[
        \|g\|_{D_\beta}^2
        \geq\sum_{k\geq0}
        \frac{|\langle g,vA^{kd}B\rangle|^2}
             {\|vA^{kd}B\|_{D_\beta}^2}
        \gtrsim|\langle g,vB\rangle|^2
        \sum_{k\geq0}\frac1{(k+1)^\beta}.
    \]
    Since the series diverges for $\beta\leq1$,
    $\langle g,vB\rangle=0$.
    Taking $v=z_1^{j_1}z_2^{j_2}a$, we obtain $g\perp[p]$.
    Since the restriction of the norm to functions of $z_2$
    alone is equivalent to the $D_\beta(\mathbb D)$ norm,
    one-variable cyclicity gives polynomials $p_j(z_2)$ such that
    $p_jp\to1$ in $D_\beta(\mathbb D^2)$.
    Thus $p$ is cyclic, and $g=0$.

    Conversely, assume $\beta>1$. Some irreducible factor $h$
    of $f$ has infinitely many zeros on $\mathbb T^2$.
    If $h$ is a constant multiple of $\zeta-z_i$, it is noncyclic
    by the preceding proposition. Otherwise, by
    \cite[Main Theorem]{Beneteau2016_bidisk},
    $h$ is noncyclic in $\mathfrak D_{\beta/2}(\mathbb D^2)$.
    Since
    \[
        (j_1+1)^{\beta/2}(j_2+1)^{\beta/2}
        \leq(j_1+j_2+2)^\beta,
        \qquad
        \|\cdot\|_{\mathfrak D_{\beta/2}}
        \leq\|\cdot\|_{D_\beta},
    \]
    $h$ is also noncyclic in $D_\beta(\mathbb D^2)$.
    Finally, $[f]\subseteq[h]\neq D_\beta(\mathbb D^2)$,
    so $f$ is noncyclic.
\end{proof}

\section{Open Problems} \label{section: open problem}

Theorem~\ref{Bidisk_Results}, part~\ref{Bidisk_Results_a},
leaves the range $3/2<\beta\leq2$ undecided for polynomials
with a nonempty finite zero set on $\mathbb T^2$. At the
time of submission, this was open even for the following
polynomial. Resolving this case would complete the
characterization of cyclic polynomials in $D_\beta(\mathbb D^2)$.

\begin{OpenProblem}
    For $ \frac{3}{2} < \beta \leq 2$, determine whether the polynomial $f(z) = 1 - \frac{z_1 + z_2}{2}$ is cyclic in $D_\beta(\mathbb{D}^2)$.
\end{OpenProblem}
While this manuscript was under consideration, Nailwal and
Zalar \cite{NailwalZalar2026} answered this question
affirmatively. Together with the results of Section~5,
their result completes the characterization of cyclic
polynomials in $D_\beta(\mathbb D^2)$.

For the complex ellipsoids, Theorem~\ref{Ellipsoid_Results}
settles the case of a nonempty finite boundary zero set.
To complete the characterization of cyclic polynomials
in $D_\beta(\Omega_p)$, it remains to treat boundary zero
sets of real dimension one. The following question asks
for an extension of the corresponding result on the unit
ball \cite{Kosinski_Vavitsas2023}.

\begin{OpenProblem}
    Let $f$ be a polynomial that does not vanish on $\Omega_p$ and vanishes on a set of real dimension one on $\partial \Omega_p$. Determine under what conditions $f$ is cyclic in $D_\beta(\Omega_p)$.
\end{OpenProblem}

\section*{Acknowledgements}
The author would like to express sincere gratitude to Professors Łukasz Kosiński and Pascal Thomas, the author’s PhD advisors, for their invaluable guidance and ongoing support. The author is particularly grateful to Professor Kosiński for generously sharing the original idea and to Professor Thomas for his substantial help in developing it further. The author also thanks Ahmed Yekta Okten and Athanasios Beslikas for insightful mathematical discussions. The author acknowledges the Institut de Mathématiques de Toulouse and the Department of Mathematics at Jagiellonian University for providing a stimulating research environment. Finally, the author thanks the anonymous referee for constructive comments that improved the clarity and presentation of the paper.

\end{document}